\documentclass[12pt]{amsart}  

\usepackage{bm}
\usepackage{bbold}
\usepackage{amscd} 
\usepackage{amsmath,amssymb}
\usepackage{mathtools}
\usepackage{mathrsfs}
\newtheorem{theorem}{Theorem}[section]
\newtheorem{lemma}[theorem]{Lemma}
 
\theoremstyle{definition}
\newtheorem{definition}[theorem]{Definition}

\newtheorem{remark}[theorem]{Remark} 
\numberwithin{equation}{section}

\def\dom{\text{\rm dom}}
\def\ran{\text{\rm ran}}

\def\RE{\mathbb R}

\def\ph*{\phi_\star}

\def\be{\begin{equation}}
\def\ee{\end{equation}}

\def\max{{\rm max}}

\def\-{{\rm in}}
\def\+{{\rm ex}}

\def\ve{\varepsilon}

\title[Sign-indefinite transmission conditions and near-zero 
resonances]{Sign-indefinite transmission conditions and near-zero eigenvalues and
resonances }
\author{A. Mantile, A. Posilicano}
\address{A. Mantile
\newline Laboratoire de Math\'{e}matiques de Reims, UMR 9008 CNRS, Universit\'{e} de Reims Champagne-Ardenne, Moulin de la Housse BP 1039, 51687
Reims, France.
\newline CNRS LYSM IRL 2019}
\address{A. Posilicano
\newline DiSAT, Sezione di Matematica, Universit\`a dell'Insubria, via Valleggio 11, I-22100
Como, Italy}
\email{andrea.mantile@univ-reims.fr}
\email{andrea.posilicano@unisubria.it}

\begin{document}

\begin{abstract} We consider  the linear operator for metamaterials modeling  $$L_{c,d }=(\chi_{\RE^{3}\backslash\overline \Omega}+c)\nabla\!\cdot (\chi_{\RE^{3}\backslash\overline \Omega}+d)\nabla\,,$$ where the material parameters $c>0$ and $d$ are constant on each connected component of the bounded domain $\Omega$, $d$ is negative on at least one component of $\Omega$ and transmission boundary conditions hold at the interfaces. It turns out that  $L_{c,d }$ has, as in the sign-definite case, an absolutely continuous spectrum equal to $[0,+\infty)$; however, contrary to the latter, it has a non-empty point spectrum; in particular, we find conditions on the parameters so that zero is a threshold eigenvalue.
In the case $\Omega$ is connected, we study how such a threshold eigenvalue behaves under $\varepsilon$-small analytic perturbations of the parameters. We show that, depending on the position of the number $\frac12\,\frac{1+d(\ve)}{1-d(\ve)}$ relative to the nearest eigenvalue of the Neumann-Poincar\'e operator of the interface, the zero eigenvalue splits into two branches (analytic in $\varepsilon^{1/2}$), consisting either of two branches of resonances, or of one branch of negative eigenvalues and one of resonances.
\end{abstract}

\maketitle

\section{Introduction and main results}

\subsection{The modeling operator $L_{c,d}$}
Let us consider the domains
\begin{equation}
\Omega\equiv\Omega_{1}\cup\dots\cup\Omega_{n}\,, \qquad \Omega_{\mathrm{ex}}:=\mathbb R^{3}\backslash\overline\Omega\,,\label{Omega_def}%
\end{equation}
where the open, bounded and simply connected domains $\Omega_{\ell}%
\subset\mathbb{R}^{3}$, $\ell=1,\dots,n$, are disjoints with $\mathcal{C}^{2}$
boundaries $\Gamma_{1},\dots,\Gamma_{n}$. The lateral trace maps at
$\Gamma_{\ell}$ are denoted by $\gamma_{j,\ell}^{\mathrm{in/ex}}$,
$j=0,1$ ($j=0$ for the Dirichlet trace, $j=1$ for the Neumann trace
respectively). The respective traces and jumps are then defined by%
\begin{equation}
\gamma_{j,\ell}:=\frac{1}{2}\,(\gamma_{j,\ell}^{\mathrm{in}}+\gamma_{j,\ell
}^{\mathrm{ex}})\,,\qquad\left[  \gamma_{j}\right]  :=\gamma_{j,\ell
}^{\mathrm{ex}}-\gamma_{j,\ell}^{\mathrm{in}}\,,\qquad j=0,1\,. \label{global}%
\end{equation}
The lateral traces $\gamma_{j}^{\mathrm{in/ex}}$ at the global boundary
$\Gamma=\Gamma_{1}\cup\dots\cup\Gamma_{n}$ are given by the direct sums
\[
\gamma_{j}^{\mathrm{in/ex}}u:=\oplus_{\ell=1}^{n}\gamma_{j,\ell}%
^{\mathrm{in/ex}}u\,,\quad j=0,1\,;
\]
the traces and jumps $\gamma_{j}$ and $[\gamma_{j}]$ at $\Gamma$ are then
defined accordingly to (\ref{global}).
\par
Denoting by $\chi_{S}$ the characteristic function of a set $S$, we introduce the piecewise-constant functions supported on $\Omega$
\begin{equation}
c:=\sum_{\ell=1}^{n}
c_{\ell}\,\chi_{\Omega_{\ell}}\,,\qquad c_{\ell}>0\,,\label{c_def}
\end{equation}
\begin{equation}
d:=\sum_{\ell=1}^{n}
d_{\ell}\,\chi_{\Omega_{\ell}}\,,\qquad d_{\ell}\in\mathbb{R}\backslash
\left\{  -1,0\right\}  \,,  \label{d_def}%
\end{equation}
where
\begin{equation}
\left.
\begin{array}
[c]{l}%
d_{\ell}<0\text{ {for at least an index }}{\ell\in}
\left\{  1,...,n\right\}
\emph{\,,}\\
\\
c_{\ell}d_{\ell}\neq1\ 
\text{{for all }}{\ell\in}
\left\{
1,...,n\right\}  \emph{\,.}%
\end{array}
\right.  
\label{d_sign_ass}%
\end{equation}
We define the operator in $L^{2}(\mathbb{R}^{3})$
\begin{equation}
L_{c,d}u(x):=-\big((  \chi_{\Omega_{\mathrm{ex}}}+c)  \nabla\!\!\cdot\!(  \chi_{\Omega_{\mathrm{ex}}%
}+d) \nabla\big) u(x)\equiv\left\{
\begin{array}
[c]{l}%
-\Delta u(x)\qquad\qquad x\in\Omega_{\mathrm{ex}}\,,\\
-c_{\ell}d_{\ell}\,\Delta u(x)\qquad x\in\Omega_{\ell}\,,
\end{array}
\right.  \label{L_cd_def}%
\end{equation}
with $c$-independent domain (here and below we refer to the next subsection for notation)
\begin{equation}
{\mathrm{dom}}(L_{c,d})=\mathcal{D}_{d}:=\left\{  u\in H^{1}(\mathbb{R}^{3})\cap L_{\Delta}%
^{2}(\mathbb{R}^{3}\backslash\Gamma):\gamma_{1,\ell}^{\mathrm{ex}}u=d_{\ell
}\,\gamma_{1,\ell}^{\mathrm{in}}u\,,\quad\ell=1,...,n\,\right\}  \,.
\label{L_d_dom}%
\end{equation}
In the following, $L^{2}(\mathbb{R}^{3},{c}^{-1})$ ($\cong L^{2}(\mathbb{R}^{3})$ as Banach space) denotes the Hilbert space
of square integrable functions with scalar product
\[
\langle f,g\rangle_{L^{2}(\mathbb{R}^{3}\!,\,{c}^{-1})}:=\int_{\mathbb{R}^{3}%
}\overline{f}(x)g(x)\,(  \chi_{\Omega_{\+}%
}(  x)  +c(  x))\,  dx\,.
\]
Using (\ref{L_cd_def}), the
boundary conditions entering in the definition \eqref{L_d_dom} and the Green identities (and by $\gamma_{0,\ell}^{\+}u=\gamma_{0,\ell}^{\-}u=\gamma_{0,\ell}u$ for any $u\in H^{1}(\RE^{3})$), one gets, for any $u,v\in
\dom(L_{c,d})$, denoting by $u_{\+}$ and $u_{\ell}$ the restriciton of $u$ to $\Omega_{\+}$ and to $\Omega_{\ell}$ respectively, and by $\Delta_{\Omega_{\+}}^{\text{\textrm{max}}}$ and  $\Delta_{\Omega_{\ell}}^{\text{\textrm{max}}}$  the maximal
Laplacians,
 \begin{align}
&  \left.  \langle u,L_{c,d}v\rangle _{L^{2}(\mathbb{R}^{3}%
,{c}^{-1})}=\langle u,L_{1,d}v\rangle _{L^{2}(\mathbb{R}^{3})}
\right.
\nonumber\\
&  \left. 
=\langle u_{\+},-\Delta_{\Omega_{\+}}^{\max}v_{\+}\rangle_{L^{2}(\Omega_{\+})}+\sum_{\ell=1}^{n}
\langle u_{\ell},-d_{\ell}\,\Delta_{\Omega_{\ell}}^{\max}v_{\ell}\rangle_{L^{2}(\Omega_{\ell})}\right.
\nonumber\\
&  \left.  =\langle -\Delta_{\Omega_{\+}}^{\max}u_{\+},v_{\+}\rangle_{L^{2}(\Omega_{\+})}-\sum_{\ell=1}^{n}\big(  \langle \gamma_{0,\ell}u,\gamma_{1,\ell
}^{\mathrm{ex}}v\rangle _{\frac12,-\frac12}-\langle \gamma_{1,\ell
}^{\mathrm{ex}}u,\gamma_{0,\ell}v\rangle_{-\frac12,\frac12} \big)
\right. \nonumber\\
&  \left.  +\sum_{\ell=1}^{n}\langle -d_{\ell}\Delta_{\Omega_{\ell}}^{\max}u_{\ell},v_{\ell}\rangle_{L^{2}(\Omega_{\ell})}+\sum_{\ell=1}^{n}d_{\ell}\big(  \langle \gamma_{0,\ell}u,\gamma_{1,\ell
}^{\mathrm{in}}v\rangle _{\frac12,-\frac12}-\langle \gamma_{1,\ell
}^{\mathrm{in}}u,\gamma_{0,\ell}v\rangle_{-\frac12,\frac12} \big)
\right. \label{int_part}
\\
&  \left.  =
\langle -\Delta_{\Omega_{\+}}^{\max}u_{\+},v_{\+}\rangle_{L^{2}(\Omega_{\+})}-\sum_{\ell=1}^{n}\big(  \langle \gamma_{0,\ell}u,\gamma_{1,\ell
}^{\mathrm{ex}}v\rangle _{\frac12,-\frac12}-\langle \gamma_{1,\ell
}^{\mathrm{ex}}u,\gamma_{0,\ell}v\rangle_{-\frac12,\frac12} \big)
\right. \nonumber\\
&  \left.  +\sum_{\ell=1}^{n}\langle -d_{\ell}\Delta_{\Omega_{\ell}}^{\max}u_{\ell},v_{\ell}\rangle_{L^{2}(\Omega_{\ell})}+\sum_{\ell=1}^{n}\big(  \langle \gamma_{0,\ell}u,\gamma_{1,\ell
}^{\mathrm{ex}}v\rangle _{\frac12,-\frac12}-\langle \gamma_{1,\ell
}^{\mathrm{ex}}u,\gamma_{0,\ell}v\rangle_{-\frac12,\frac12} \big)
  \right. \nonumber
  \\
&  \left.  =
\langle -\Delta_{\Omega_{\+}}^{\max}u_{\+},v_{\+}\rangle_{L^{2}(\Omega_{\+})}+
\sum_{\ell=1}^{n}\langle -d_{\ell}\Delta_{\Omega_{\ell}}^{\max}u_{\ell},v_{\ell}\rangle_{L^{2}(\Omega_{\ell})}
=\langle
u,L_{1,d}v\rangle _{L^{2}(\mathbb{R}^{3})}\right. \nonumber\\
&  \left.  =\langle L_{c,d}u,v\rangle _{L^{2}(\mathbb{R}^{3}%
,{c}^{-1})}\,.\right.  \nonumber
\end{align}
Hence, $L_{c,d}$ is a symmetric operator in $L^{2}(\mathbb{R}^{3},{c}^{-1})$.
Moreover, the quadratic form associated to $L_{c,d}$ is%
\begin{equation}
F_{c,d}(  u)  =\Vert \nabla u_{\+}\Vert^{2} _{L^{2}(\Omega_{\+})}+\sum_{\ell=1}^{n}
d_{\ell}\Vert \nabla u_{\ell}\Vert^{2}  _{L^{2}(\Omega_{\ell})}. \label{EM_form}%
\end{equation}
By (\ref{d_sign_ass}), $F_{c,d}$ is not semibounded and the self-adjointness
of $L_{c,d}$ (as an operator in $L^{2}(\mathbb{R}^{3},{c}^{-1})$) cannot be
derived using (standard) quadratic forms arguments. Here, we show the
self-adjointness of $L_{c,d}$ by providing its resolvent (see Theorem \ref{Theorem_Krein_EM}).

When $u\in\mathcal{C}_{0}^{\infty}(\mathbb{R}^{3})$ and $v\in\mathcal{D}_{d}$,
the same computation as in (\ref{int_part}) yields (here $\langle\cdot,\cdot\rangle$ denotes the $\mathcal{C}_{0}^{\infty}(\mathbb{R}^{3}))\,$-$\,\mathcal{C}_{0}^{\infty}(\mathbb{R}^{3})'$ duality)%
\[
\left\langle u,L_{c,d}v\right\rangle=\left\langle -{\Delta
}u,v\right\rangle +\sum_{\ell=1}^{n}
\langle (1-c_{\ell}d_{\ell})  
\Delta_{\Omega_{\ell}}^{\max} u_{\ell},v_{\ell}\rangle -
\sum_{\ell=1}^{n}\langle \gamma_{0,\ell}u,[  \gamma_{1,\ell}]  v\rangle
_{\frac12,-\frac12}\,.
\]
Since
\[
\gamma_{1,\ell}^{\mathrm{ex}}u=\gamma_{1,\ell}u+\frac{1}{2}\,[\gamma_{1,\ell
}]u\,,\qquad\gamma_{1,\ell}^{\mathrm{in}}u=\gamma_{1,\ell}u-\frac{1}%
{2}\,[\gamma_{1,\ell}]u\,,
\]
the boundary conditions in \eqref{L_d_dom} are equivalent to
\begin{equation}
\lbrack\gamma_{0,\ell}]u=0\,,\qquad\left[  \gamma_{1,\ell}\right]
u=-2\,\frac{1-d_{\ell}}{1+d_{\ell}}\,\gamma_{1,\ell}u\,,\qquad\ell
=1,\dots,n\,. \label{transmission_bc}%
\end{equation}
Using (\ref{transmission_bc}) and duality, one obtains 
\[
\langle \gamma_{0,\ell}u,[  \gamma_{1,\ell}]  v\rangle
_{\frac12,-\frac12}=
2\,\frac{1-d_{\ell}}{1+d_{\ell}}\left\langle u,\gamma_{0,\ell}^{\ast}%
\gamma_{1,\ell}\psi\right\rangle _{\mathbb{R}^{3}}\,.
\]
This leads to the distributional identity%

\begin{equation}
L_{c,d}=-\Delta_{\mathbb{R}^{3}}+(1-cd)1_{\Omega}^{\ast}\Delta_{\Omega
}^{\text{\textrm{max}}}1_{\Omega}-2\,\frac{1-d}{1+d}\,\gamma_{0}^{\ast
}\,\gamma_{1}\,, \label{L_cd_weak}%
\end{equation}
where $\Delta_{\mathbb{R}^{3}}$ is the distributional Laplacian on
$\mathbb{R}^{3}$. Hence, in such a weak sense, $L_{c,d}$ is an additive
perturbation of the Laplacian with a relatively $\Delta$-bounded term plus a
distributional term supported on $\Gamma$ (since each $\gamma_{0,\ell}^{\ast
}\in\mathscr{B}(  H^{-1/2}(  \Gamma_{\ell})  ,H^{-1}
(\mathbb{R}^{3}))  $ has values in distributions supported on
$\Gamma_{\ell}$).

\subsection{\label{Sec_Notation}Notation}{\,}
\par\noindent$\bullet$ 
${\mathbb{C}%
}_{\pm}:=\left\{  z\in{\mathbb{C}}:\pm\operatorname{Im}z)>0\right\}  $.
\par\noindent$\bullet$ 
$D_{r}:=\left\{  z\in\mathbb{C}:\left\vert
z\right\vert <r\right\}  $,  
$B_{r}=\left\{
x\in\mathbb{R}^{3}:\left\vert x\right\vert <r\right\}  $. 
\par\noindent$\bullet$ 
$a\lesssim b$ means $a\leq c\,b$ for some constant $c$.
\par\noindent$\bullet$ 
Given a map $0<\varepsilon\mapsto u(\varepsilon)\in
X$, $X$ a Banach space, $u(\varepsilon)=O(\varepsilon^{\lambda})$ means $\Vert
u(\varepsilon)\Vert_{X}\lesssim\varepsilon^{\lambda}$. 
\par\noindent$\bullet$ 
$R_{\kappa}\equiv(-\Delta-\kappa^{2})^{-1}$ denotes the free
resolvent with kernel $\frac{e^{i\kappa|x-y|}}{4\pi|x-y|}$. 
\par\noindent$\bullet$
$S\!L_{\kappa}:=R_{\kappa}\gamma_{0}^{\ast}$ and
$D\!L_{\kappa}:=R_{\kappa}\gamma_{1}^{\ast}$ denote the single- and
double-layer operators for the global boundary $\Gamma$. 
\par\noindent$\bullet$ 
$S_{\kappa}:=\gamma_{0}S\!L_{\kappa}$, $K_{\kappa}:=\gamma
_{1}S\!L_{\kappa}$, $D_{\kappa}:=\gamma_{1}D\!L_{\kappa}$, $N\!P_{\kappa
}:=\gamma_{0}D\!L_{\kappa}\equiv K_{-\bar{\kappa}}^{\ast}$;  $S\!L_{\ell,\kappa}$, $D\!L_{\ell,\kappa}$, $S_{\ell,\kappa}$,
$D_{\ell,\kappa}$, $N\!P_{\ell,\kappa}$, $K_{\ell,\kappa}$ denote the same
kind of operators on $\Gamma_{\ell}$. 
\par\noindent$\bullet$ 
$N_{\kappa}:=1_{\Omega}R_{\kappa}1_{\Omega}^{\ast}$; in
particular, $N_{0}:=1_{\Omega}R_{0}1_{\Omega}^{\ast}$ is the Newton potential
operator in $L^{2}(\Omega)$;  $N_{\ell,\kappa}$
denote the same kind of operators on $\Omega_{\ell}$. 
\par\noindent$\bullet$ 
$1_{\Omega
}:L^{2}(\mathbb{R}^{3})\rightarrow L^{2}(\Omega)$ is the restriction map to
$\Omega$; its adjoint $1_{\Omega}^{\ast}:L^{2}(\Omega)\rightarrow
L^{2}(\mathbb{R}^{3})$ represents the extension by zero.
\par\noindent$\bullet$ 
$H^{s}(\mathcal{O})$ are the Sobolev spaces on the open set
$\mathcal{O}\subseteq\mathbb{R}^{3}$; $H^{s}(\Gamma)$ are the Sobolev spaces
on the boundary $\Gamma$.
\par\noindent$\bullet$ 
$\langle\cdot,\cdot
\rangle_{-s,s}$ denotes the dual pairing, induced from the $L^{2}(\Gamma
)$-scalar product, anti-linear with respect to the first variable, between
$H^{-s}(\Gamma)$ and $H^{s}(\Gamma)$.
\par\noindent$\bullet$ 
$\Delta
_{\mathcal{O}}$ is the distributional Laplacian on $\mathcal{O}\subseteq
\mathbb{R}^{3}$.
\par\noindent$\bullet$ 
$H_{\Delta}^{s}(\mathcal{O}%
):=\left\{  u\in H^{s}(\mathcal{O}):\Delta_{\mathcal{O}}u\in L^{2}%
(\mathcal{O})\right\}  $; $L_{\Delta}^{2}(\mathcal{O})\equiv H_{\Delta}%
^{0}(\mathcal{O})$.
\par\noindent$\bullet$ 
$H_{\alpha}^{s}(\RE^{3})  $, $L_{\alpha}^{2}(\RE^{3})  \equiv H_{\alpha}^{0}(\RE^{3})$, denotes the weighted $H^{s}$-spaces with weight
$\left\langle x\right\rangle ^{\alpha}:=(1+|x|^{2})^{\alpha/2}$;  $L_{\alpha}^{2}(\RE^{3})  \equiv H_{\alpha}^{0}(\RE^{3})$.%
\par\noindent$\bullet$ 
$\Delta_{\mathcal{O}}^{\mathrm{max}}:=\Delta_{\mathcal{O}}|L_{\Delta
}^{2}(\mathcal{O})$ is the maximal Laplacian on $L^{2}(\mathcal{O}%
)$.
\par\noindent$\bullet$
$\mathscr{B}(X,Y)$ is the Banach space of bounded linear operators on $X$ to
$Y$; $\mathscr{B}(X)\equiv\mathscr{B}(X,X)$; the corresponding operator  is denoted by $\Vert\cdot\Vert_{X,Y}$ and when no confusion arises we
use the abbreviated notation $\Vert \cdot\Vert $.
\par\noindent$\bullet$ 
${\mathfrak{S}}_{\infty}(X,Y)\subset\mathscr{B}(X,Y)$ denotes
the set of compact operators.
\par\noindent$\bullet$ 
${\varrho}(A)$ and
$\sigma(A)={\mathbb{C}}\backslash{\varrho}(A)$ denote the resolvent set and
the spectrum of the operator $A$. For a selfadjoint $A$, $\sigma_{p}(A)$, $\sigma_{ac}(A)$, $\sigma_{sc}(A)$, $\sigma_{ess}(A)$,
$\sigma_{disc}(A)$, denote the point,  absolutely continuous,
singular continuous, essential and discrete spectra.

\subsection{Previous works on related models}

Operators of the type $L_{c,d}$, $c=1$, $d<0$, were considered in a confining
domain $U\supset\Omega$, $\Omega$ made of a single connected component, and
Dirichlet boundary conditions at the boundary of $U$. The strategy to achieve
self-adjointness in the not-critical case, i.e. when $d\neq-1$, was settled in
\cite{BBDauRa CRASP 99}. Critical values of the transmission parameter where
later considered for particular geometries in \cite{BheKre JAM 14} and
\cite[Chp. 8]{Schm Thesis 14}. In \cite{CaPaPo JAMA 19} the problem of
self-adjointness was reformulated in terms of self-adjoint extensions of a
symmetric differential operator. Using the boundary triplet method, this was
reduced to the study of the essential self-adjointness of a pseudodifferential
operator on $\Gamma$. The authors recover the known results for a not-critical
contrast, while present new results in the case of a critical contrast,
showing how the domain of self-adjointness depends on the dimension and on the
geometry. A more general setting, involving sign-changing transmission
conditions in a non piece-wise homogeneous model with confinement, was
considered in \cite{BenPanka}.

The operators (\ref{L_cd_def})-(\ref{L_d_dom}) are used for the propagation of
electric and magnetic transverse modes in materials with sign-changing
parameters. Such configurations have relevant applications, including
plasmonic resonances and metamaterial designs. The lack of coercivity of
$L_{c,d}$ entails a well-posedness issue for the quasi-modes problems. This
was first addressed in \cite{BoBe JCAM 10}, \cite{BoBe ESIAM 12} using the
T-coercivity approach. A limiting absorption principle and the well-posedness
in the case of not constant contrast were obtained in \cite{Ngu JMPA 2016}
using complementing boundary conditions.

The resonances of $L_{c,d}$ are usually defined as the poles of the
meromorphic extension of a cutoff resolvent%
\begin{equation}
\kappa\mapsto\psi\big(  L_{c,d}-\kappa^{2}\big)  ^{-1}\psi\,,\qquad
\psi\in\mathcal{C}_{0}^{\infty}(  \mathbb{R}^{3})  \,,
\label{resolvent_cd_cutted_smooth}%
\end{equation}
to the second Riemann sheet, which in our $3D$ setting corresponds to
$\mathbb{C}_{-}$. It has been noticed that these poles coincide with those
$\kappa$ for which the problem%
\begin{equation}
\left\{
\begin{array}
[c]{l}%
\left(  -\Delta-\kappa^{2}\right)  u(x)=0\,,\qquad\qquad x\in\Omega_{\+}\,,\\
\left(  -cd\,\Delta-\kappa^{2}\right)  u(x)\,,\qquad\qquad\quad x\in
\Omega\,,\\
\lbrack\gamma_{0}]u=0\,,\qquad\left[  \gamma_{1}\right]  u=-2\,\frac{1-d}%
{1+d}\,\gamma_{1}u\,,
\end{array}
\right.  \label{scattering_res_eq}%
\end{equation}
has $\kappa$-outgoing solutions, namely: there exist $r_{\circ}>0$ and a compactly
supported distribution $g$ such that%
\begin{equation}
u(  x)  =\big(  (  L_{c,d}-\kappa^{2})  ^{-1}g\big)
(  x)  \,,\qquad\text{as }\left\vert x\right\vert >r_{\circ}\,.
\label{outgoing_cond}%
\end{equation}
This setting was adopted in \cite{PoVo CMP 99}, for $c,d>0$, and in \cite{CaMo
IMA 23}, in the case of sign-changing coefficients, to study the asymptotic
behaviour of the resonances with large real part.

The stationary-scattering problem%
\[
\left\{
\begin{array}
[c]{l}%
\left(  \nabla\!\!\cdot\!\left(  \frac{1}{\mu^{\mathrm{ex}}}\chi
_{\Omega_{\mathrm{ex}}}+\frac{1}{\mu^{\mathrm{in}}}%
\chi_{\Omega}\right)  \nabla+\omega^{2}\right)  u=0\,,\qquad\text{in
}\mathbb{R}^{3}\backslash\Gamma\,,\\
\left(  \Delta+\omega^{2}\right)  u_{\circ}=0\,,\qquad\text{in }\mathbb{R}%
^{3}\backslash\Gamma\,,\\
u^{\mathrm{sc}}=u-u_{\circ}\ \text{satisfies the Sommerfeld radiation condition
}\\
\lbrack\gamma_{0}]u=0\,,\qquad\frac{1}{\mu^{\mathrm{ex}}}\gamma_{1}%
^{\mathrm{ex}}u=\frac{1}{\mu^{\mathrm{in}}}\gamma_{1}^{\mathrm{in}}u\,,
\end{array}
\right.
\]
was considered in \cite{Amm et al ARMA 2017} in the quasi-static limit 
$\omega\rightarrow0_{+}$. An enhancement of the scattered field
$u^{\mathrm{sc}}$ is there observed depending on the material's parameters
$\mu^{\mathrm{in/ex}}$, namely when%
\[
\lambda:=\frac{1}{2}\frac{\mu^{\mathrm{ex}}+\mu^{\mathrm{in}}}{\mu
^{\mathrm{ex}}-\mu^{\mathrm{in}}}%
\]
is close to some eigenvalue of the Neumann-Poincar\'{e} operator $N\!P_{0}$ on
the interface $\Gamma$. In particular, when $\left\vert \lambda-\lambda
_{j}\right\vert $ is small for a simple $\lambda_{j}\in\sigma_{p}\left(
N\!P_{0}\right)  $, the far-field 
behaves at low-frequencies as (see \cite[Theorem 2.1]{Amm et al ARMA 2017})
\[
\frac{1}{\lambda-\lambda_{j}+{O}\left(  \omega^{2}\right)  }%
\]
and a blow-up is reached at $\omega=0$ and $\lambda=\lambda_{j}$. Such
regimes, where the scattered field concentrates at the interface, are referred
to as surface plasmons. A more refined analysis in convex domains has shown
that the surface modes concentrate at the regions of maximal curvature of the
boundary, \cite{AmmLiu 2}.

Let $v$ and $\rho$ refer to positive sets of coefficients $\left\{  v_{\ell
}\right\}  $, $\left\{  \rho_{\ell}\right\}  $. A study through a
Kre\u{\i}n-like resolvent formula of acoustic operators of the kind (be aware
that here our definition of $A_{v,\rho}$ differs by a minus sign from the one
in \cite{MaPo Acoustic 26})%
\begin{equation}
A_{v,\rho}:=-\left(  \chi_{\Omega_{\mathrm{ex}}}+v^{2}\rho
\right)  \nabla\!\!\cdot\!\left(  \chi_{\Omega_{\mathrm{ex}}}+\rho^{-1}\right)  \nabla\,,\qquad
\dom(  A_{v,\rho})  =\mathcal{D}_{1/\rho}\,,\ \rho
>0,\ v>0\,, \label{A_acou}%
\end{equation}
was developed in \cite{MaPo Acoustic 26} by adapting the abstract framework in
\cite{MaPo JST 23}. In particular, it was shown that the resolvent formula%
\begin{equation}
\left(  A_{v,\rho}-\kappa^{2}\right)  ^{-1}=R_{\kappa}+%
\begin{bmatrix}
R_{\kappa}1_{\Omega}^{\ast} & S\!L_{\kappa}%
\end{bmatrix}
\left(  Q_{\kappa}^{\mathrm{acoustic}}\right)  ^{-1}%
\begin{bmatrix}
({v}^{2}-1)1_{\Omega}^{\ast}\Delta_{\Omega}^{\mathrm{max}}1_{\Omega}R_{\kappa
}\\
2\,\frac{\rho-1}{\rho+1}\gamma_{1}R_{\kappa}%
\end{bmatrix}
\label{Krein_acou}%
\end{equation}
holds, where the block operator matrix
\[
Q_{\kappa}^{\mathrm{acoustic}}=%
\begin{bmatrix}
v^{2}+\left(  v^{2}-1\right)  \kappa^{2}N_{\kappa} & \left(  v^{2}-1\right)
\kappa^{2}1_{\Omega}S\!L_{\kappa}\\
-2\,\frac{\rho-1}{\rho+1}\,\gamma_{1}^{\mathrm{in}}N_{\kappa} & 1-2\,\frac
{\rho-1}{\rho+1}\,K_{\kappa}%
\end{bmatrix}
\]
turns out to have a bounded inverse for any $\kappa\in\mathbb{C}_{+}$, see
\cite[Theorem 4.6]{MaPo Acoustic 26}. Here $R_{\kappa}$ is the resolvent of
the free Laplacian in $\mathbb{R}^{3}$, $N_{\kappa}$ its compression to
$L^{2}(  \Omega)  $, $S\!L_{\kappa}=R_{\kappa}\gamma_{0}^{\ast}$
the single-layer potential at $\Gamma$ and $K_{\kappa}=\gamma_{1}R_{\kappa
}\gamma_{0}^{\ast}$ is the $L^{2}(\Gamma)$-adjoint of the Neumann-Poincar\'{e}
operator. The existence of $\kappa\mapsto\left(  Q_{\kappa}^{\mathrm{acoustic}%
}\right)  ^{-1}$ was achieved using Schur's complement and meromorphic
Fredholm theory, after proving the existence of the inverse operators
\[
\left(  1-2\,\frac{\rho-1}{\rho+1}\,K_{0}\right)  ^{-1}\in\mathscr{B}%
(H^{-1/2}(\Gamma))\,.
\]
The latter is consequence of%
\[
2\ \frac{\rho-1}{\rho+1}\in\left(  -2,2\right)  \,,
\]
and $\sigma(K_{0})=\sigma(N\!P_{0})\subset\left[  -1/2,1/2\right)  $
(a known spectral property of the Neumann-Poincar\'{e} operator), see
\cite[Lemmata 4.5 and 7.1]{MaPo Acoustic 26}.

\subsection{Aims and results}

Due to (\ref{d_sign_ass}), $L_{c,d}$ cannot be rephrased as an acoustic
operator of the type considered above. However, the formal analogy of
(\ref{L_cd_def}) and (\ref{A_acou}) suggests to define the $Q$-function of the
couple $\left(  L_{c,d},-\Delta\right)  $ as%
\begin{equation}
Q_{\kappa}(c,d)=%
\begin{bmatrix}
cd+({cd}-1)\kappa^{2}N_{\kappa} & (cd-1)\kappa^{2}1_{\Omega}S\!L_{\kappa}\\
-\mu(  {d})  \,\gamma_{1}^{\mathrm{in}}N_{\kappa} & 1-\mu\left(
{d}\right)  \,K_{\kappa}%
\end{bmatrix}
\label{Q_cd}%
\end{equation}
where $\mu(  {d})  $ is the piecewise constant function on $\Gamma$
defined by%
\begin{equation}
\mu(  {d})  :=\sum_{\ell=1}^{n}
\mu_{\ell}(  d_{\ell})  \,\chi_{\Gamma_{\ell}}\,,\qquad\mu_{\ell
}(  d_{\ell})  :=2\,\frac{1-d_{\ell}}{1+d_{\ell}}\,.
\label{mu_d_def}%
\end{equation}
We notice that $d_{\ell}\in\mathbb{R}\backslash\left\{  -1,0\right\}  $
implies: $\mu_{\ell}(  d_{\ell})  \in\mathbb{R}\backslash\left\{
\pm2\right\}  $. In the present setting $d$ is not sign-definite and the
argument used in \cite[Sec. 4]{MaPo Acoustic 26} to prove the existence of
$Q_{\kappa}^{-1}(c,d)$ can only be used partially. In particular, if $d_{\ell
}\in\left(  -\infty,0\right)  \backslash\left\{  -1\right\}  $ one has
\[
\mu_{\ell}(  d_{\ell})  \in\left(  -\infty,-2\right)  \cup\left(
2,+\infty\right)  \,.
\]
Since $\sigma(K_{\ell,0})\subset\left[  -1/2,1/2\right)  $, there are specific
values of $d_{\ell}<0$ at which $1\in\sigma(\mu_{\ell}(  d_{\ell})
K_{\ell,0})$. In such a case $(1-\frac{1-d}{1+d}\,K_{0})$ has a not-trivial
kernel and, as a consequence, $Q_{\kappa}(c,d)$ is not invertible at
$\kappa=0$.

The main assumption in our work is 
\begin{description}
\item[\emph{H0}] Either $1\notin\sigma\left(  \mu(  {d})
\,K_{0}\right)  $ or $1\in\sigma\left(  \mu(  {d})  \,K_{0}\right)
$ is a simple eigenvalue.
\end{description}

In what follows, such an assumption is tacitly assumed (it will be explicitly recalled
in Sections 2 and 3). 

A part of our analysis is carried on in the particular case where all the $d_{\ell}$'s have
the same value (obviously, this is always the case whenever $\Omega$ reduces to a single connected component). In such a
case, \emph{H0} is replaced by 

\begin{description}
\item[\emph{H0.1}] Let $d_{\ell}=d_{\circ}$ for any index $\ell$ and set $\mu(
d_{\circ}):=2\,\frac{1-d_{\circ}}{1+d_{\circ}}$. Then $\mu^{-1}(
d_{\circ})  \notin\sigma(K_{0})$ or $\mu^{-1}(  d_{\circ})  \in
\sigma(K_{0})$ is a simple eigenvalue.
\end{description}
It is worth to recall that the not-zero eigenvalues of the
Neumann-Poincar\'{e} operators $K_{0}$ are generically simple in the
sense of the Baire category, see \cite{AnKanMiPu 23}. Hence, H0.1 holds for generic domains, out of those with some symmetry (as, e.g., the
balls). 

In this work we prove the existence of $\kappa\mapsto Q_{\kappa}%
^{-1}(c,d)$ as a meromorphic function in $\mathbb{C}$ with values in
$\mathscr{B}(  L^{2}(  \Omega)  \oplus H^{-1/2}(\Gamma))  $. In particular, we show in Theorem
\ref{Theorem_Q_full} that $\kappa\mapsto Q_{\kappa}^{-1}\left(  c,d\right)  $
has a pole in $\kappa=0$ whenever $1\in\sigma\left(  \mu(  {d})
\,K_{0}\right)  $. Building on results from \cite{MaPo Acoustic 26}, in
Sections 4, we prove the self-adjointness of $L_{c,d}$ providing a resolvent
formula of the kind (\ref{Krein_acou}) and a corresponding  limiting
absorption principle (LAP in the following), see Theorem \ref{Theorem_Krein_EM}%
, Lemma \ref{Lemma_LAP} and Remark \ref{rem-LAP}. The spectral analysis is
developed in Section 5, Theorem \ref{theorem_spectrum}, and specialized in
Theorem \ref{Theorem_zero_eigenvalue} where we discuss the existence of the
zero eigenvalue in the setting \emph{H0.1}. The outcomes of this analysis are
resumed in the next Theorem.

\begin{theorem}
\label{Th1}$L_{c,d}$ is self-adjoint in $L^{2}(\mathbb{R}^{3},c^{-1})$ and the
resolvent formula%
\be
R_{\kappa}(c,d):=\left(  L_{c,d}-\kappa^{2}\right)  ^{-1}=R_{\kappa}+%
\begin{bmatrix}
R_{\kappa}1_{\Omega}^{\ast} & S\!L_{\kappa}%
\end{bmatrix}
Q_{\kappa}^{-1}(c,d)%
\begin{bmatrix}
(cd-1)\Delta_{\Omega}^{\mathrm{max}}1_{\Omega}R_{\kappa}\\
\mu(d)\,\gamma_{1}R_{\kappa}%
\end{bmatrix}\label{Krein_L_cd}
\ee
holds for any $\kappa\in\mathbb{C}_{+}$ such that $\kappa^{2}\in{\varrho
}(L_{c,d})$. The limit%
\[
R_{\kappa}(c,d):=\lim_{\delta\rightarrow0^{+}}R_{\kappa+i\delta}%
(c,d)\,,\qquad\kappa\in\mathbb{R}\backslash\left\{  0\right\}  \,,
\]
exists in $\mathscr{B}(H_{\alpha}^{0}(\mathbb{R}^{3}),H_{-\alpha}%
^{1}(\mathbb{R}^{3}))$ for $\alpha>3/2$. The resolvent formula (\ref{Krein_EM}%
) holds in the limit and, for any $u\in H_{\alpha}^{0}(\mathbb{R}^{3})$ one
has%
\[
\left(  \left(  -\Delta_{\Omega_{\mathrm{ex}}}\right)
\oplus\left(  -cd\Delta_{\Omega}\right)  -\kappa^{2}\right)  R_{\kappa
}(c,d)u=u\,,\qquad\kappa\in\mathbb{R}\backslash\{0\}\,.
\]
Furthermore,
\[
\sigma_{p}(L_{c,d})\cap(0,+\infty)=\varnothing\,,\quad\sigma_{ess}%
(L_{c,d})=[0,+\infty)\,,\quad\sigma_{disc}(L_{c,d})\text{ accumulates at
}-\infty\,,
\]%
\[
\sigma_{sc}(L_{c,d})=\varnothing\,,\quad\sigma_{ac}(L_{c,d})=[0,+\infty)\,.
\]
Finally, if $d_{\ell}=d_{\circ}\in(-\infty,0)\backslash\{  -1\}
$ for any index $\ell$ and $\mu^{-1}(d_{\circ})\in\sigma(K_{0})$, then
$0\in\sigma_{p}(L_{c,d})$.
\end{theorem}

A framework for the resonances is presented in Section 6: in Theorem
\ref{teo-res} we show that the set of resonance $\mathrm{Res}(L_{c,d})$
is given by the 
characteristic points of the $Q$-function, i.e, 
$$
\mathrm{Res}(L_{c,d})=\{\kappa^{2}\in\mathbb C:\ker(  Q_{\kappa}(
c,d))  \neq\{  0\}\,,\ \kappa\in\mathbb C_{-}\}\,. 
$$
In Sections 7 we
introduce the assumption \emph{H1} and \emph{H2} (replacing \emph{H0}). Here,
$\Omega\ $is simply connected (hence, we avoid the use of the index $\ell$)
and the material parameters $c(\varepsilon)$,
$d(\varepsilon)$ are analytic functions of
$\varepsilon$, such that
\begin{align*}
&  \left.  c(\varepsilon)=c_{0}+c_{1}\varepsilon
+{O}(\varepsilon^{2})\,,\quad c_{0}>0\,,\ c_{1}\not =0\,,\right.  \\
& \\
&  \left.  d(\varepsilon)=\frac{\lambda_{\circ}-\rho(\varepsilon)-1/2}{\lambda
_{\circ}-\rho(\varepsilon)+1/2}\,,\qquad\rho(\varepsilon
)=\rho_{0}\varepsilon+{O}(\varepsilon^{2})\,,\quad\rho_{0}%
\neq0\,,\right.  \\
& \\
&  \left.  \lambda_{\circ}\in\sigma(K_{0})\backslash\left\{  0,-1/2\right\}
\text{ is a simple eigenvalue.}\right.
\end{align*}
This entails
\[
\tilde{d}(  \varepsilon)  \equiv\mu(  d(\varepsilon))
=2\,\frac{1-d(\varepsilon)}{1+d(\varepsilon)}\,,\qquad\tilde{d}^{-1}(
\varepsilon)  =\lambda_{\circ}-\rho(  \varepsilon)  \,.
\]
Hence, the small parameter $\varepsilon$ controls the distance of $\tilde
{d}^{-1}(\varepsilon)$ from $\sigma(K_{0})$. In this setting, we look for  the
near-zero resonances of the operators
\[
\left.  L(\varepsilon)\equiv-(  \chi_{\Omega_{\+}}+c(\varepsilon))  \nabla\!\cdot
\!(  \chi_{\Omega_{\mathrm{ex}}}+d(\varepsilon
))  \nabla\,.\right.
\]
Since $\tilde{d}^{-1}(0)=\lambda_{\circ}\in\sigma(K_{0})$ is
not-degenerate, the origin is an eigenvalue $L(0)$. Using an analytic
implicit function argument and a Lyapunov-Schmidt reduction, we show that,
under the analytic perturbation of the coefficients, such zero eigenvalue
splits into a couple of poles $\kappa_{\pm}(  \varepsilon)  $ of
the $Q$-function.
In particular, for $\varepsilon\in\left(  0,\varepsilon_{\circ}\right)  $ and
depending on the sign of $\rho_{0}/\lambda_{\circ}$, the points $\kappa_{\pm
}(  \varepsilon)  $ correspond either to a couple of resonances 
 or an eigenvalue and
a resonance. The result of this analysis, presented in Theorem
\ref{Teo-res} and Remark \ref{Rem-res} is next resumed.

\begin{theorem}
\label{Th2}Let $\varphi_{\circ}\in H^{-1/2}(\Gamma)$ denotes the normalized
eigenvector for $\lambda_{\circ}\in\sigma(K_{0})\backslash\left\{  0,-1/2\right\}
$. There exists $\varepsilon_{0}>0$ and a unique couple of functions
$(0,\varepsilon_{0})\ni\varepsilon\mapsto\kappa_{\pm}\left(  \varepsilon
\right)  $, analytic with respect to $\sqrt{\varepsilon}$, such that:\medskip
\newline$i)$ if $\rho_{0}/\lambda_{\circ}<0$, then
\[
\kappa_{\pm}(  \varepsilon)  =\pm\sqrt{2}\,\left\vert \frac
{{\rho_{0}}}{\lambda_{\circ}}\right\vert ^{1/2}
\Vert\varphi_{\circ}\Vert_{\hat{H}^{-1/2}(\Gamma)}
\sqrt{\varepsilon}\,\left(  1+{O}(\varepsilon
)\right)  +{O}(  \varepsilon^{2})
\]
and $\kappa_{\pm}^{2}(  \varepsilon)  $ are both resonances of
$L(\varepsilon)$;
\medskip\newline$ii)$ if $\rho_{0}/\lambda_{\circ}>0$, then
\[
\kappa_{\pm}(\varepsilon)=\pm i\sqrt{2}\,\left(  \frac{{\rho_{0}}}{\lambda
_{\circ}}\right)  ^{1/2}\Vert\varphi_{\circ}\Vert_{\hat{H}^{-1/2}(\Gamma)}%
\sqrt{\varepsilon}\,\left(  1+{O}(\varepsilon)\right)  +{O}%
\left(  \varepsilon^{2}\right)
\]
and $\kappa_{+}^{2}(  \varepsilon)  $ is a negative eigenvalue,
while $\kappa_{-}^{2}(  \varepsilon)  $ is a resonances of
$L(\varepsilon)$.
\end{theorem}

The Kre\u{\i}n-type resolvent formula (\ref{Krein_L_cd}) for $
L_{c,d}  $ and the relations in Theorems \ref{Th1} linking the
spectral properties of $L_{c,d}$ to the kernel of its $Q$-function are new and
extends our previous works \cite{MaPo JST 23}, \cite{MaPo Acoustic 26} to the
case of sign-changing transmission conditions. The derivation of the LAP
through the Kre\u{\i}n's formula, in Lemma \ref{Lemma_LAP} (together with Remark
\ref{rem-LAP}), simplifies in our setting the analysis provided in \cite{Ngu
JMPA 2016}. The spectral nature of $\kappa=0$ for sign-changing transmission
problems with piecewise constant coefficients was discussed in \cite{BheKre
JAM 14}, \cite{Schm Thesis 14} and \cite{CaPaPo JAMA 19} in the case of
a Dirichlet confinement. In this regard, our work extend such an analysis to the
unbounded case. The operator $L(\varepsilon)$ is of the type considered in
\cite{Amm et al ARMA 2017} and in similar works on plasmonic resonances. The
result provided in Theorem \ref{Th2} confirms the expectation that, in the
asymptotic regime of $\tilde{d}^{-1}(\varepsilon)\sim\lambda_{\circ}\in
\sigma(K_{0})$, the scattering solutions of the problem%
\[
\left(  L(\varepsilon)-\omega^{2}\right)  u=0\,,
\]
are enhanced at low frequency, since Theorem \ref{Th2} predicts that
$\kappa(\varepsilon)\sim0$ is a pole of the scattering amplitude (not
introduced here). We point out that the analyticity and spectral nature of
$\kappa(\varepsilon)$, were not
discussed in \cite{Amm et al ARMA 2017} and we are not aware that such an analysis has been carried out so far. 

\section{\label{Sec_Transmission}The transmission operator}

Let us recall that $\mathbb{C}\ni\kappa\mapsto N\!P_{\kappa}$ is an analytic
family of compact operators on $H^{1/2}( \Gamma) $, see \cite[Lemma
5.3]{MaPoSi JMPA 22}. Hence, for each $\kappa\in\mathbb{C}$, $\sigma\left(
N\!P_{\kappa}\right)  \backslash\left\{  0\right\}  $ is made of a at most
countably-infinite set of eigenvalues of finite multiplicity possibly
accumulating at $\kappa=0$. In
particular, for the the Neumann-Poincar\'{e} operator one has $\sigma(
N\!P_{0}) \subset\left[  -1/2,1/2\right)  $ and $\kappa=-1/2$ is the
eigenvalue corresponding to the constant eigenfunction, see e.g. in \cite[Sec.
5]{AnKanMiPu RevMath 21}.

The operators $K_{\kappa}:=\gamma_{1}R_{\kappa}\gamma_{0}^{\ast}$ are related
to the $L^{2}(\Gamma)$-adjoint of $N\!P_{\kappa}$ by $K_{\kappa}%
=N\!P_{-\bar{\kappa}}^{\ast}$ (we use $R_{\kappa}^{\ast}=R_{-\bar{\kappa}}$).
The mapping properties of $N\!P_{\kappa}$ imply that $\mathbb{C}\ni
\kappa\mapsto K_{\kappa}$ is an analytic family of compact operators in
$H^{-1/2}(\Gamma)$. The Plemelj symmetrization principle%
\begin{equation}
N\!P_{0}S_{0}=S_{0}K_{0} \label{symm_id}%
\end{equation}
implies $\sigma(N\!P_{0})=\sigma(K_{0})$. Let us introduce the scalar
product
\begin{equation}
\lbrack\phi,\varphi]_{-\frac{1}{2}}:=\langle S_{0}\phi,\varphi\rangle
_{\frac{1}{2},-\frac{1}{2}}\,,\qquad\phi,\varphi\in H^{-1/2}(\Gamma)\,,
\label{mod_prod}%
\end{equation}
where
\[
S_{0}:H^{-1/2}(\Gamma)\rightarrow H^{1/2}(\Gamma)\,,\qquad S_{0}:=\gamma
_{0}S\!L_{0}\,,
\]
is an isomorphism and is self-adjoint with respect to the $H^{1/2}$-$H^{-1/2}$
duality. We denote by $\hat{H}^{-1/2}(  \Gamma)  $ the space
$H^{-1/2}(  \Gamma)  $ equipped with the scalar product
$[\cdot,\cdot]_{-\frac{1}{2}}$. From \eqref{symm_id} there follows
\begin{align*}
&  \left.  \left[  K_{0}\phi,\varphi\right]  _{-\frac{1}{2}}=\left\langle
S_{0}K_{0}\phi,\varphi\right\rangle _{\frac{1}{2},-\frac{1}{2}}=\left\langle
\phi,N\!P_{0}S_{0}\varphi\right\rangle _{-\frac{1}{2},\frac{1}{2}}=\left\langle
\phi,S_{0}K_{0}\varphi\right\rangle _{-\frac{1}{2},\frac{1}{2}}\right. \\
&  \left.  =\left\langle S_{0}\phi,K_{0}\varphi\right\rangle _{\frac{1}%
{2},-\frac{1}{2}}=\left[  \phi,K_{0}\varphi\right]  _{-\frac{1}{2}}\right.
\end{align*}
and thus $K_{0}$ is self-adjoint in $\hat{H}^{-1/2}(  \Gamma)  $.
Let us further recall that $N\!P_{0}1=-1/2$; hence, (\ref{symm_id}) implies%
\[
K_{0}\left(  S_{0}^{-1}1\right)  =-\frac{1}{2}\left(  S_{0}^{-1}1\right)  \,.
\]
In particular, by the selfadjointness of $K_{0}$ in $\hat{H}^{-1/2}\left(
\Gamma\right)  $, for any eigencouple $\left(  \lambda_{\circ},\varphi_{\circ}\right)
$ of $K_{0}$ with $\lambda_{\circ}\neq-1/2$ one has,%
\begin{equation}
\left\langle 1,\varphi_{\circ}\right\rangle _{\frac{1}{2},-\frac{1}{2}%
}=\left\langle S_{0}\left(  S_{0}^{-1}1\right)  ,\varphi_{\circ}\right\rangle
_{\frac{1}{2},-\frac{1}{2}}=\left[  \left(  S_{0}^{-1}1\right)  ,\varphi
_{\circ}\right]  _{-\frac{1}{2}}=0\,. \label{zero_mean}%
\end{equation}

We next focus on the operator pencil%
\begin{equation}
\kappa\mapsto(  1-\mu(  d)  \,K_{\kappa})  \,,
\label{Op_pencil}%
\end{equation}
where $\mu(  d)  $ is defined by (\ref{mu_d_def}). Let us recall
the assumption \emph{H0}:

\begin{description}
\item[\emph{H0}] Assume that either $1\notin\sigma\left(  \mu(
{d})   K_{0}\right)  $ or $1\in\sigma\left(  \mu(  d)
 K_{0}\right)  $ is a simple eigenvalue.
\end{description}
\begin{remark}
\label{Remark_inverse}If%
\[
d_{\ell}=d_{\circ}\in\mathbb{R}\backslash\left\{  -1,0\right\}  \quad\forall
\ \ell=1,...,n\,,
\]
in particular, if $\Omega$ is made of a single connected component, then assumption \emph{H0} rephrases as 
$$
\text{$\mu^{-1}(  d_{\circ})
\notin\sigma(  K_{0})  $ or $\mu^{-1}(  d_{\circ})
\in\sigma(  K_{0})  $ is a simple eignvalue.}
$$
\end{remark}

\begin{theorem}
\label{Theorem_inverse}If \emph{H0} holds, then
\begin{equation}
\mathbb{C}\ni\kappa\mapsto(  1-\mu(  d)  \,K_{\kappa})
^{-1} \label{Op_pencil_inv}%
\end{equation}
is a finitely meromorphic $\mathscr{B}(H^{-1/2}(\Gamma))$-valued map. In
particular:\medskip\newline$i)$ If $1\notin\sigma(  \mu({d})   K_{0})  $, then $\kappa\mapsto(  1-\mu(d)  \,K_{\kappa})  ^{-1}$ is analytic at $\kappa=0$.\medskip
\newline$ii)$ If $1\in\sigma(  \mu(  d)   K_{0})  $,
then $\kappa\mapsto(  1-\mu(  d)  \,K_{\kappa})  ^{-1}$
has a pole at $\kappa=0$.
\end{theorem}

\begin{proof}
Since $\kappa\mapsto(  1-\mu(  d)  \,K_{\kappa})  $ is
analytic and $\kappa\mapsto\mu(  d)  \,K_{\kappa}$ compact-operator
valued, the meromorphic Fredholm theory applies, see \cite[Theorem
XIII.13]{ReSi IV}. Hence, either the map (\ref{Op_pencil_inv}) is finitely
meromorphic in $\mathbb{C}$ or $(1-\mu(  d)  \,K_{\kappa})$ is
nowhere invertible. We distinguish the two cases: 1) $1\notin\sigma\left(
\mu(  d)   K_{0}\right)  $; 2) $1\in\sigma\left(  \mu\left(
d\right)   K_{0}\right)  $ is a simple eigenvalue.

1) If $1\notin\sigma\left(  \mu(  d)   K_{0}\right)  $, then
$\ker\left(  1-\mu(  d)   K_{0}\right)  =\varnothing$. Hence
$(  1-\mu(  d)  \,K_{\kappa})  ^{-1}\in\mathscr{B}%
(H^{-1/2}(\Gamma))$ exists by the Fredholm alternative and, by continuity of
$\kappa\mapsto(  1-\mu(  d)  \,K_{\kappa})  $, the
operators $(  1-\mu(  d)  \,K_{\kappa})  ^{-1}%
\in\mathscr{B}(H^{-1/2}(\Gamma))$ exist in an open neighbourhood of $\kappa=0$.
Henceforth (\ref{Op_pencil_inv}) is finitely meromorphic in $\mathbb{C}$.

2) Let $1\in\sigma\left(  \mu(  d)   K_{0}\right)  $ be simple and
assume that $(  1-\mu(  d)  \,K_{\kappa})  $ is not
invertible at any $\kappa\in\mathbb{C}$. We show that this statement leads to
a contradiction. By the Fredholm alternative, $\ker\left(  1-\mu\left(
d\right)  \,K_{\kappa}\right)  \neq\left\{  0\right\}  $, which is equivalent
to say that, for each $\kappa$, there exists a non-trivial solution
$\varphi(\kappa)\in H^{-1/2}(\Gamma)$ of the eigenvalue problem%
\[
\mu(  d)  \,K_{\kappa}\varphi(\kappa)=\varphi(  \kappa)
\,.
\]
We work in a disk ${D}_{\delta}$ with $\delta>0$ small and denote%
\[
{D}_{\delta}^{+}:=\left\{  \kappa\in\mathbb{C}_{+}:\left\vert \kappa
\right\vert <\delta\right\}  \,.
\]
Since $\kappa\mapsto\mu(  d)  \,K_{\kappa}$ is analytic and
$1\in\sigma\left(  \mu(  d)   K_{0}\right)  $ is not-degenerate,
the Kato-Rellich theorem applies (e.g. \cite[Theorem XII.8]{ReSi IV}).
Therefore, the eigencouple $(1,\varphi(  \kappa)  )$ is analytic in
$D_{\delta}$. Let us define $\psi(\kappa):=S\!L_{\kappa}\varphi(\kappa)$. By
the trace maps properties of the single-layer operator, one gets
\begin{equation}
\left[  \gamma_{0}\right]  \psi(\kappa)=0\,,\qquad\left[  \gamma_{1}\right]
\psi(\kappa)=-\varphi(  \kappa)  =-\mu(  d)
\,\gamma_{1}\psi(\kappa)\,. \label{transmission_bc_id}%
\end{equation}
For $\kappa\in{D}_{\delta}^{+}$, the well-known mapping properties (see e.g.
in \cite[Theorem 6.12 and successive remarks]{McLe}), the first jump condition
in (\ref{transmission_bc_id}) and the identity%
\begin{equation}
-\Delta S\!L_{\kappa}\varphi(\kappa)|\mathbb{R}^{3}\backslash\Gamma=\kappa^{2}S\!L_{\kappa}\varphi
(\kappa)|\mathbb{R}^{3}\backslash\Gamma\,,\label{SL_id}%
\end{equation}
imply $\psi(\kappa)\in H^{1}(\mathbb{R}^{3})\cap L_{\Delta}^{2}(\mathbb{R}%
^{3}\backslash\Gamma)$. Define $L_{1,d}$ as in (\ref{L_cd_def})-(\ref{L_d_dom})
with $c_{\ell}=1$ for all $\ell$ and $d=\left(  d_{1},...d_{\ell}\right)  $.
Then, (\ref{transmission_bc_id}) corresponds to the boundary conditions
(\ref{transmission_bc}). It follows that $\psi(\kappa)\in\mathcal{D}_{d}$ and,
using (\ref{L_cd_def}) and (\ref{SL_id}), one obtains%
\[
L_{1,d}\psi(\kappa)|\Omega_{\mathrm{ex}}=\kappa^{2}%
\psi(\kappa)|\Omega_{\mathrm{ex}}\,,\qquad L_{1,d}%
\psi(\kappa)|\Omega_{\ell}=\kappa^{2}d_{\ell}\psi(\kappa)|\Omega_{\ell}\,.
\]
Therefore, for any $\kappa\in{D}_{\delta}^{+}$, one gets the identities%
\begin{align*}
\left\langle \psi(\kappa),L_{1,d}\psi(  \kappa)  \right\rangle
_{L^{2}(\mathbb{R}^{3})}=  &  \left\langle \psi(\kappa),-{\Delta}\psi
(\kappa)\right\rangle _{L^{2}(\Omega_{\mathrm{ex}})}+%
{\textstyle\sum\nolimits_{\ell=1}^{n}}
\left\langle \psi(  \kappa)  ,-d_{\ell}\,{\Delta}\psi
(\kappa)\right\rangle _{L^{2}(\Omega_{\mathrm{ex}})  }\\
=  &  \kappa^{2}\left(  \left\Vert \psi(\kappa)\right\Vert _{L^{2}(\Omega_{\mathrm{ex}})  }^{2}+%
{\textstyle\sum\nolimits_{\ell=1}^{n}}
d_{\ell}\,\left\Vert \psi(\kappa)\right\Vert _{L^{2}(  \Omega_{\ell
})  }^{2}\right)
\end{align*}
and%
\begin{align*}
\left\langle L_{1,d}\psi(\kappa),\psi(  \kappa)  \right\rangle
_{L^{2}(\mathbb{R}^{3})}=  &  \left\langle -{\Delta}\psi(\kappa),\psi
(\kappa)\right\rangle _{L^{2}(  \Omega_{\mathrm{ex}})  }+%
{\textstyle\sum\nolimits_{\ell=1}^{n}}
\left\langle -d_{\ell}\,{\Delta}\psi(  \kappa)  ,\psi
(\kappa)\right\rangle _{L^{2}(  \Omega_{\mathrm{ex}})  }\\
=  &  \overline{\kappa^{2}}\left(  \left\Vert \psi(\kappa)\right\Vert
_{L^{2}(  \Omega_{\mathrm{ex}})  }^{2}+%
{\textstyle\sum\nolimits_{\ell=1}^{n}}
d_{\ell}\,\left\Vert \psi(\kappa)\right\Vert _{L^{2}(  \Omega_{\ell
})  }^{2}\right)  \,.
\end{align*}
Since $\varphi(\kappa)\neq0$, from the injectivity of $S\!L_{\kappa}$ there
follows $\psi(  \kappa)  \neq0$ for any $\kappa$. For $\kappa\in
{D}_{\delta}^{+}\backslash i\mathbb{R}_{+}$, one has $\kappa^{2}\neq
\overline{\kappa^{2}}$ and the symmetry of $L_{1,d}$ (as an operator in
$L^{2}(\mathbb{R}^{3})$) implies%
\begin{equation}
\left\Vert \psi(  \kappa)  \right\Vert _{L^{2}(
\Omega_{\mathrm{ex}})  }^{2}+%
{\textstyle\sum\nolimits_{\ell=1}^{n}}
d_{\ell}\,\left\Vert \psi(  \kappa)  \right\Vert _{L^{2}(\Omega
)}^{2}=0\,. \label{abv}%
\end{equation}
By (\ref{d_sign_ass}), there exists $d_{\ell_{\circ}}<0$ for some $\ell_{\circ}$.
Choosing ${D}_{\delta}\ni\kappa_{\circ}>0$ and taking the limit ${D}_{\delta}%
^{+}\backslash i\mathbb{R}_{+}\ni\kappa\rightarrow\kappa_{\circ}$ in (\ref{abv}),
one gets%
\[
\Vert \psi(\kappa_{\circ})\Vert _{L^{2}(\Omega_{\+})  }^{2}\leq\vert d_{\ell_{\circ}%
}\vert\,\Vert \psi(\kappa_{\circ})\Vert _{L^{2}(\Omega_{\tilde{\ell}})  }^{2}\,.
\]
Hence, $\psi(\kappa_{\circ})\in L^{2}(\mathbb{R}^{3})$. Let $r>0$ such that
$\Omega\subset B_{r}$; by (\ref{SL_id}), $\psi_{\circ}(\kappa_{\circ})$ solves the
Helmholtz equation outside $B_{r}$; the Rellich estimate implies $\left.
\psi(\kappa_{\circ})\right\vert B_{r}=0$ (see, e.g., \cite[Corollary 4.8]{Leis}).
Then, by the unique continuation property, $\left.  \psi\left(  \kappa
_{\circ}\right)  \right\vert \Omega_{\mathrm{ex}}=0$. Since
$\psi\left(  \kappa_{\circ}\right)  $ fulfills (\ref{transmission_bc_id}), both
the inner traces $\ $at $\Gamma$, $\gamma_{j}^{\mathrm{in}}\psi\left(
\kappa_{\circ}\right)  $, $j=0,1$, vanish. From (\ref{SL_id}) there follows%
\[
\left\{
\begin{array}
[c]{l}%
\left(  \Delta+\kappa^{2}\right)  \psi(\kappa_{\circ})|\Omega=0\,,\\
\gamma_{0}^{\mathrm{in}}\psi(\kappa_{\circ})=\gamma_{1}^{\mathrm{in}}\psi
(\kappa_{\circ})=0\,.
\end{array}
\right.
\]
This implies $\left.  \psi(\kappa_{\circ})\right\vert \Omega=0$. In conclusion,
$\psi(\kappa_{\circ})=S\!L_{\kappa}\varphi(\kappa_{\circ})=0$, which is in
contradiction with $\varphi(\kappa_{\circ})\neq0$.
\end{proof}

We denote by $\Sigma_{d}\subseteq\mathbb{C}$ the domain of analyticity of
$\kappa\mapsto(  1-\mu(  d)  \,K_{\kappa})  ^{-1}$,
i.e.,%
\begin{equation}
\Sigma_{d}\equiv\left\{  \kappa\in\mathbb{C}:\left(  1-\mu(  d)
\,K_{\kappa}\right)  ^{-1}\in\mathscr{B}(H^{-1/2}(\Gamma))\right\}  \,.
\label{sigma-d}%
\end{equation}

\begin{lemma}
If \emph{H0} holds and $\kappa\in\Sigma_{d}\cap\mathbb{C}_{+}$, then
\begin{equation}
\mathcal{D}_{d}=\left\{  u=u_{\circ}+S\!L_{\kappa}\left(  1-\mu(  d)
\,K_{\kappa}\right)  ^{-1}\mu(  d)  \,\gamma_{1}u_{\circ}\,,\qquad
u_{\circ}\in H^{2}(\mathbb{R}^{3})\right\}  \,. \label{dom_id}%
\end{equation}

\end{lemma}

\begin{proof}
If $u\in\mathcal{D}_{d}$ then $u\in H^{1}(\mathbb{R}^{3})\cap L_{\Delta}%
^{2}(\mathbb{R}^{3}\backslash\Gamma)$; by $L_{\Delta}^{2}(\mathbb{R}%
^{3}\backslash\Gamma)\equiv{\mathrm{dom}}(\Delta_{\Omega_{\+}}^{\mathrm{max}}\oplus\Delta_{\Omega}^{\mathrm{max}})$, for
$\operatorname{Im}\kappa>0$ one has (see \cite[Lemma 4.2]{MaPoSi JDE 2016})
\[
u=u_{\circ}+S\!L_{\kappa}\varphi-D\!L_{\kappa}\phi\,,\qquad u_{\circ}\in
H^{2}(\mathbb{R}^{3})\,,\ \varphi\in H^{-3/2}(\Gamma)\,,\ \phi\in
H^{-1/2}(\Gamma)\,.
\]
The boundary conditions (\ref{transmission_bc}), corresponding to%
\begin{equation}
\lbrack\gamma_{0}]u=0\,,\qquad\left[  \gamma_{1}\right]  u=-\mu\left(
d\right)  \,\gamma_{1}u\,, \label{transmission_bc_global}%
\end{equation}
imply $\phi=0$ and $\varphi=\mu(  d)  \,\gamma_{1}u$. This gives
\[
u=u_{\circ}+\mu(  d)  \,S\!L_{\kappa}\gamma_{1}u\,,
\]
where $\gamma_{1}u$ is determined, for any fixed $u_{\circ}$, by the boundary
equation%
\[
(  1-\mu(  d)  \,K_{\kappa})  \gamma_{1}u=\gamma
_{1}u_{\circ}\,.
\]
For $\kappa\in\Sigma_{d}$ there follows%
\[
\gamma_{1}u=(  1-\mu(  d)  \,K_{\kappa})  ^{-1}%
\gamma_{1}u_{\circ}\,,
\]
which leads to%
\begin{equation}
u=u_{\circ}+S\!L_{\kappa}^{-1}(  1-\mu(  d)  \,K_{\kappa})
\mu(  d)  \,\gamma_{1}u_{\circ}\,. \label{dom_shape}%
\end{equation}
Conversely, if $\kappa\in\Sigma_{d}\cap\mathbb{C}_{+}$ and $u$ is defined by
(\ref{dom_shape}) with $u_{\circ}\in H^{2}(\mathbb{R}^{3})$, then $u\in
H^{1}(\mathbb{R}^{3})\cap L_{\Delta}^{2}(\mathbb{R}^{3}\backslash\Gamma)$ (by
the single-layer mapping properties) and, by taking the Neumann trace of
(\ref{dom_shape}), one has%
\[
\gamma_{1}u=\left(  1+S\!L_{\kappa}\left(  1-\mu(  d)  \,K_{\kappa
}\right)  \mu(  d)  \right)  \,\gamma_{1}u_{\circ}\,.
\]
Then, by%
\begin{align*}
&  \left.  \left[  \gamma_{1}\right]  u=-\left(  1-\mu(  d)
\,K_{\kappa}\right)  \mu(  d)  \,\gamma_{1}u_{\circ}=-(1-\mu\left(
d\right)  \,K_{\kappa}+\mu(  d)  \,K_{\kappa})\left(  1-\mu\left(
d\right)  \,K_{\kappa}\right)  ^{-1}\mu(  d)  \,\gamma_{1}%
u_{\circ}\right. \\
&  \left.  -(  1-\mu(  d)  \,K_{\kappa})  \left(
1-\mu(  d)  \,K_{\kappa}\right)  ^{-1}\mu(  d)
\,\gamma_{1}u_{\circ}-\mu(  d)  \,K_{\kappa}^{-1}\left(  1-\mu\left(
d\right)  \,K_{\kappa}\right)  \mu(  d)  \,\gamma_{1}u_{\circ}\right.
\\
&  \left.  -\mu(  d)  \,(\left(  1-K_{\kappa}\left(  1-\mu\left(
d\right)  \,K_{\kappa}\right)  ^{-1}\mu(  d)  \right)  \gamma
_{1}u_{\circ}\right.
\end{align*}
there follows $u\in\mathcal{D}_{d}$.
\end{proof}

\section{The inverse of the $Q$-function}

We consider the operators $Q_{\kappa}( c,d) $ defined in (\ref{Q_cd}). Our
goal is to prove the existence of ${\mathbb{C}}\ni\kappa\mapsto Q_{\kappa
}^{-1}( c,d) $ as an operator-valued meromorphic function. 
We premise a result regarding the analytic mapping properties of the cutoff
free resolvent; we use the notation $M_{\psi}$ to denote the multiplication
operator by the function $\psi$.

\begin{lemma}
\label{Lemma_R_k_analytic}Let $\psi\in\mathcal{C}_{c}^{\infty}(\mathbb{R}%
^{3})$, $s\in\left[  0,2\right]  $. Then $\mathbb{C}\ni\kappa\mapsto
M_{\psi}R_{\kappa}M_{\psi}$ is a $\mathscr{B}(H^{s-2}(\mathbb{R}^{3}%
),H^{s}(\mathbb{R}^{3}))$-valued analytic map.
\end{lemma}

\begin{proof}
By \cite[Theorem 3.1]{DyaZwo},
\[
\left\Vert M_{\psi}R_{\kappa}M_{\psi}\right\Vert _{L^{2}(\mathbb{R}^{3}%
),H^{2}(\mathbb{R}^{3})}\lesssim(1+|\kappa|)\,e^{L\,\text{\textrm{Im}}%
(\kappa)_{-}}\,,
\]
where $L>\text{diam}(\text{supp}(\phi))$ and $\text{\textrm{Im}}(\kappa)_{-}$
denotes the negative part. By replacing $\kappa$ with $\bar{k}$, and by
duality, one gets the same upper bound for $\left\Vert M_{\psi}R_{\kappa
}M_{\psi}\right\Vert _{H^{-2}(\mathbb{R}^{3}),L^{2}(\mathbb{R}^{3})}$. Hence,
by interpolation,
\[
\left\Vert M_{\psi}R_{\kappa}M_{\psi}\right\Vert _{H^{s-2}(\mathbb{R}%
^{3}),H^{s}(\mathbb{R}^{3})}\lesssim(1+|\kappa|)\,e^{L\,\text{\textrm{Im}%
}(\kappa)_{-}}\,.
\]
Therefore, the $\mathscr{B}(H^{s-2}(\mathbb{R}^{3}),H^{s}(\mathbb{R}^{3}%
))$-valued map $\mathbb{C}\ni\kappa\mapsto M_{\psi}R_{\kappa}M_{\psi}$ is
locally bounded and so, see, e.g., \cite[page 365]{Kato}), it is enough to
show that $\mathbb{C}\ni\kappa\mapsto\langle M_{\psi}R_{\kappa}M_{\psi
}\,u,v\rangle$ is analytic for any $u$ and $v$ belonging to dense sets in
$H^{2-s}(\mathbb{R}^{3})$ and $H^{-s}(\mathbb{R}^{3})$ respectively. Hence,
the proof is concluded by \cite[equation (5.8)]{MaPoSi JMPA 22}, which implies
the analyticity of the map
\[
\mathbb{C}\ni\kappa\mapsto\langle M_{\psi}R_{\kappa}M_{\psi}\,u,\varphi
\rangle,\qquad\quad u\in H^{s-2}(\mathbb{R}^{3})\,,\ \varphi\in\mathcal{C}%
_{c}^{\infty}(\mathbb{R}^{3})\,.
\]
\hfill
\end{proof}
\begin{lemma}
\label{Lemma_Q_Fredholm} The map ${\mathbb{C}}\ni\kappa\mapsto Q_{\kappa
}(c,d)$ is $\mathscr{B}(  L^{2}(  \Omega)  \oplus
H^{-1/2}(  \Gamma)  )  $-valued analytic. In particular,
\begin{equation}
{\mathbb{C}}\ni\kappa\mapsto%
\begin{bmatrix}
N_{\kappa} & 1_{\Omega}S\!L_{\kappa}\\
\gamma_{1}^{\mathrm{in}}N_{\kappa} & \,K_{\kappa}%
\end{bmatrix}
\label{Fredholm_cond}%
\end{equation}
is an analytic ${\mathfrak{S}}_{\infty}(  L^{2}(  \Omega)
\oplus H^{-1/2}(  \Gamma)  )  $-valued map.
\end{lemma}

\begin{proof}
We consider the operator entries of \eqref{Fredholm_cond}. Let $r>0$ be large
enough so that $\Omega\subset\subset B_{r}=\left\{  x:\left\vert x\right\vert
<r\right\}  $ and let $\psi\in\mathcal{C}_{c}^{\infty}(\mathbb{R}^{3})$ be
such that $\psi|B_{r}=1$. Since $N_{\kappa}=1_{\Omega}M_{\psi}R_{\kappa
}M_{\psi}1_{\Omega}^{\ast}$, by Lemma \ref{Lemma_R_k_analytic}, one has
$N_{\kappa}\in\mathscr{B}(L^{2}(\Omega),H^{2}(\Omega))$ and $\gamma
_{1}^{\mathrm{in}}N_{\kappa}\in\mathscr{B}\left(  L^{2}(  \Omega)
,H^{1/2}(  \Gamma)  \right)  $ for any $\kappa\in\mathbb{C}$. Then,
by the compact embeddings $H^{2}(\Omega)\hookrightarrow L^{2}(\Omega)$ and
$H^{1/2}(\Gamma)\hookrightarrow H^{-1/2}(\Gamma)$, one gets $N_{\kappa}%
\in{\mathfrak{S}}_{\infty}(L^{2}(\Omega))$ and $\gamma_{1}^{\mathrm{in}%
}N_{\kappa}\in{\mathfrak{S}}_{\infty}(L^{2}(\Omega),H^{1/2}(\Gamma))$. By
Lemma \ref{Lemma_R_k_analytic}, $\kappa\mapsto N_{\kappa}$ and $\kappa
\mapsto\gamma_{1}^{\mathrm{in}}N_{\kappa}$ are analytic in $\mathscr{B}%
(L^{2}(  \Omega)  )$ and $\mathscr{B}(L^{2}(  \Omega)
,H^{-1/2}(\Gamma))$ respectively.

By the mapping properties of the single layer potential, see e.g. \cite[Chp.
6]{McLe}, for each $\kappa\in\mathbb{C}$ one has $1_{\Omega}S\!L_{\kappa}%
\in\mathscr{B}(H^{-1/2}(\Gamma),H^{1}(  \Omega)  )$ and from the
compact embedding $H^{1}(  \Omega)  \hookrightarrow L^{2}\left(
\Omega\right)  $ there follows that $1_{\Omega}S\!L_{\kappa}\in{\mathfrak{S}%
}_{\infty}(H^{-1/2}(\Gamma),L^{2}(\Omega))$. Moreover, since $1_{\Omega
}S\!L_{\kappa}=1_{\Omega}M_{\psi}R_{\kappa}M_{\psi}\gamma_{0}^{\ast}$, by
$\gamma_{0}^{\ast}\in\mathscr{B}(H^{-1/2}(\Gamma),H^{-1}(\mathbb{R}^{3}))$ and
Lemma \ref{Lemma_R_k_analytic}, $\kappa\mapsto1_{\Omega}S\!L_{\kappa}$ is an
analytic $\mathscr{B}(H^{-1/2}(\Gamma),H^{1}(\Omega))$-valued map.

As recalled in Section \ref{Sec_Transmission}, $\kappa\mapsto K_{\kappa}%
\in\mathscr{B}(H^{-1/2}(\Gamma))$ is analytic in $\mathbb{C}$ and with values
in the set of compact operators.
\end{proof}

We next address the problem of the inversion of $Q_{\kappa}( c,d) $.

\begin{theorem}
\label{Theorem_Q_full}If \emph{H0} holds, then $\mathbb{C}\ni\kappa\mapsto
Q_{\kappa}^{-1}(c,d)\in\mathscr{B}\left(  L^{2}(  \Omega)  \oplus
H^{-1/2}(  \Gamma)  \right)  $ is finitely-meromorphic. In
particular:\medskip\newline$i)$ If $1\notin\sigma\left(  \mu\left(
{d}\right)   K_{0}\right)  $, then $\kappa\mapsto Q_{\kappa}^{-1}(c,d)$ is
analytic at $\kappa=0$.\medskip\newline$ii)$ If $1\in\sigma\left(  \mu\left(
{d}\right)   K_{0}\right)  $, then $\kappa\mapsto Q_{\kappa}^{-1}(c,d)$ has a
pole at $\kappa=0$.
\end{theorem}

\begin{proof}
The map $\kappa\mapsto Q_{\kappa}(c,d)$ is analytic in $\mathbb{C}$ and
recasts in the form%
\begin{equation}
Q_{\kappa}(c,d)=%
\begin{bmatrix}
cd & 0\\
0 & 1
\end{bmatrix}
\left(  \,%
\begin{bmatrix}
1 & 0\\
0 & 1
\end{bmatrix}
+%
\begin{bmatrix}
\frac{1}{cd}(cd-1)\kappa^{2} & 0\\
0 & -\mu(  d)  \,
\end{bmatrix}%
\begin{bmatrix}
N_{\kappa} & 1_{\Omega}S\!L_{\kappa}\\
\gamma_{1}^{\mathrm{in}}N_{\kappa} & \,K_{\kappa}%
\end{bmatrix}
\,\right)  \,. \label{Q_cd_id}%
\end{equation}
By the result of the Lemma \ref{Lemma_Q_Fredholm}, the meromorphic Fredholm
theory (see \cite[Theorem XIII.13]{ReSi IV}) applies to the r.h.s. of
(\ref{Q_cd_id}). Hence: either $\kappa\mapsto Q_{\kappa}^{-1}\left(
c,d\right)  \in\mathscr{B}\left(  L^{2}(\Omega)\oplus H^{-1/2}\left(
\Gamma\right)  \right)  $ is finitely-meromorphic in $\mathbb{C}$ or
$Q_{\kappa}(c,d)$ is nowhere invertible. We next show that this last statement
leads to a contradiction.

Assume that $Q_{\kappa}(c,d)$ is not invertible for all $\kappa\in\mathbb{C}$.
By the Fredholm alternative, $\ker Q_{\kappa}(  c,d)  \neq\left\{
(0,0)\right\}  $, i.e., there exists a non-trivial $\left(  u(\kappa
),\varphi(\kappa)\right)  \in L^{2}(\Omega)\oplus H^{-1/2}(\Gamma)$ solving
the equation
\[
Q_{\kappa}(c,d)\left(  u(\kappa),\varphi(  \kappa)  \right)
=(0,0)\,,
\]
equivalently,
\begin{equation}
\left\{
\begin{array}
[c]{l}%
\left(  cd+(cd-1)\kappa^{2}N_{\kappa}\right)  u(\kappa)+(cd-1)\kappa
^{2}1_{\Omega}S\!L_{\kappa}\varphi(\kappa)=0\,,\\
-\mu(  d)  \,\gamma_{1}^{\mathrm{in}}N_{\kappa}u(\kappa)+\left(
1-\mu(  d)  \,K_{\kappa}\right)  \varphi(\kappa)=0\,.
\end{array}
\right.  \label{Q_ker_eq}%
\end{equation}
We next assume $\kappa\in\Sigma_{d}\backslash\{0\}$ defined in (\ref{sigma-d}%
). Notice that%
\[
\kappa\in\Sigma_{d}\quad\Leftrightarrow\quad\ker\left(  1-\mu(  d)
\,K_{\kappa}\right)  =\{0\}\,.
\]
From the second equation, we have%
\begin{equation}
\varphi(\kappa)=(  1-\mu(  d)  \,K_{\kappa})  ^{-1}%
\mu(  d)  \,\gamma_{1}^{\mathrm{in}}N_{\kappa}u(\kappa)\,.
\label{Q_ker_cond}%
\end{equation}
Since $\left(  u(\kappa),\varphi(\kappa)\right)  \neq\left(  0,0\right)  $,
this implies: $u(\kappa)\neq0$. Using (\ref{d_sign_ass}) and (\ref{Q_ker_cond}%
), the above system rephrase as%
\begin{equation}
\left\{
\begin{array}
[c]{l}%
\Psi_{\kappa}\,u(\kappa)=\frac{1}{\kappa^{2}}\,\frac{cd}{1-cd}\,u(\kappa)\,,\\
\varphi(\kappa)=(  1-\mu(  d)  \,K_{\kappa})  ^{-1}%
\mu(  d)  \,\gamma_{1}^{\mathrm{in}}N_{\kappa}u(\kappa)\,,
\end{array}
\right.  \label{Q_ker_eq_1}%
\end{equation}
where $\Psi_{\kappa}\in\mathscr{B}(L^{2}(\Omega))$ is defined by
\[
\Psi_{\kappa}:=N_{\kappa}+1_{\Omega}S\!L_{\kappa}\left(  1-\mu\left(
{d}\right)  \,K_{\kappa}\right)  ^{-1}\mu(  d)  \,\gamma
_{1}^{\mathrm{in}}N_{\kappa}\,.
\]
Therefore, $u(\kappa)$ is an eigenvectors of $\Psi_{\kappa}$. Let us introduce
$\psi(\kappa)\in L^{2}(\mathbb{R}^{3})$ defined by%
\[
\psi(\kappa):=R_{\kappa}1_{\Omega}^{\ast}u(\kappa)+S\!L_{\kappa}\left(
1-\mu(  d)  \,K_{\kappa}\right)  ^{-1}\mu(  d)
\,\gamma_{1}R_{\kappa}1_{\Omega}^{\ast}u(\kappa)u(\kappa)\,.
\]
By (\ref{dom_id}) (with $u_{\circ}=R_{\kappa}1_{\Omega}^{\ast}u(\kappa)$), one has
$\psi(\kappa)\in\mathcal{D}_{d}$. By (\ref{Q_ker_eq_1}), it results%
\begin{equation}
\psi(\kappa)|\Omega=\Psi_{\kappa}u(\kappa)\,. \label{psi_def}%
\end{equation}
Since $\mathrm{ran}(S\!L_{\kappa})\cap H^{2}(\mathbb{R}^{3})=\{0\}$ and
$u(\kappa)\neq0$, one has $\psi(\kappa)\neq0$. Furthermore,%
\[
\left.  \left(  L_{c,d}-\kappa^{2}\right)  \psi(\kappa)\right\vert
\Omega_{\mathrm{ex}}=\left.  \left(  -\Delta-\kappa
^{2}\right)  \psi(\kappa)\right\vert \mathbb{R}^{3}\backslash\overline{\Omega
}=\left.  1_{\Omega}^{\ast}u(\kappa))\right\vert \mathbb{R}^{3}\backslash
\overline{\Omega}=0\,,
\]
and
\begin{align*}
\left(  L_{c,d}-\kappa^{2}\right)  \psi(\kappa)|\Omega=  &  (-cd\,\Delta
-\kappa^{2})\psi(\kappa)|\Omega=cd\,(-\Delta-\kappa^{2})\psi(\kappa
)|\Omega-(1-cd)\,\kappa^{2}\psi(\kappa)|\Omega\\
=  &  cd\,u(\kappa)-(1-cd)\,\kappa^{2}\psi(\kappa)|\Omega=0\,.
\end{align*}
Hence, for any $\kappa\in\Sigma_{d}\backslash\left[  0,+\infty\right)  $,
$\psi(\kappa)$ is an eigenfunction of $L_{c,d}$ with eigenvalue $\kappa^{2}$.
However, since $L_{c,d}$ is symmetric, this is only possible when $\kappa
^{2}\in\mathbb{R}$.

If $1\notin\sigma\left(  \mu(  d)   K_{0}\right)  $,
the\ results in $(i)$-Theorem \ref{Theorem_inverse} applies and $\kappa\mapsto
(  1-\mu(  d)  \,K_{\kappa})  ^{-1}$ is
analytic at $\kappa=0$. The Schur complement of the second diagonal entry is
(\ref{Q_cd}) is%
\[
C_{\kappa}:=cd+({cd}-1)\kappa^{2}\left(  1+1_{\Omega}S\!L_{\kappa}\left(
1-\mu(  d)  \,K_{\kappa}\right)  ^{-1}\mu(  d)
\,\gamma_{1}^{\mathrm{in}}\right)  N_{\kappa}\,.
\]
Since $\kappa\mapsto C_{\kappa}$ is analytic in a neighborhood of $\kappa=0$
and $C_{0}$ has bounded inverse, $\kappa\mapsto C_{\kappa}^{-1}$ is analytic
in a neighborhood of $\kappa=0$ as well. Then $(i)$ follows using Schur's
complement argument.

If $1\in\sigma\left(  \mu(  d)   K_{0}\right)  $, by
$(ii)$-Theorem \ref{Theorem_inverse}, $\left(  1-\mu(  d)
 K_{0}\right)  $ has non-trivial kernel. Hence, if $\varphi\in\ker\left(
1-\mu(  d)   K_{0}\right)  $, then $(0,\varphi)\in\ker
Q_{0}(c,d)$.
\end{proof}

\section{The resolvent formula and self-adjointness}

From now on, the hypothesis \emph{H0}, ensuring the existence of $Q_{\kappa
}^{-1}(c,d)$, is assumed. Let us introduce the bounded maps%
\begin{align*}
&  \left.  \tau:H_{\Delta}^{1}(\mathbb{R}^{3}\backslash\Gamma)\rightarrow
H_{\Delta}^{1}(\mathbb{R}^{3}\backslash\Gamma)\oplus H^{1/2}(\Gamma)\oplus
H^{-1/2}(\Gamma)\,,\right. \\
&  \left.  \tau u:=\tau_{1}u\oplus\tau_{2}u\,,\qquad\tau_{1}u:=u\,,\qquad
\tau_{2}u:=\gamma_{0}u\oplus\gamma_{1}u\,,\right. \\
& \\
&  \left.  G_{\kappa}:L^{2}(\mathbb{R}^{3})\oplus H^{-1/2}(\Gamma)\oplus
H^{1/2}(\Gamma)\rightarrow H_{\Delta}^{1}(\mathbb{R}^{3}\backslash
\Gamma)\,,\right. \\
&  \left.  G_{\kappa}(u\oplus\varphi\oplus\phi):=R_{\kappa}u+S\!L_{\kappa
}\varphi+D\!L_{\kappa}\phi\right. \\
& \\
&  \left.  B_{1}(cd):=(cd-1)\Delta_{\Omega}^{\mathrm{max}}1_{\Omega}%
\in\mathscr{B}((H_{\Delta}^{1}(\mathbb{R}^{3}\backslash\Gamma)),L^{2}%
(\Omega))\,,\right. \\
& \\
&  \left.  B_{2}(d\,):=%
\begin{bmatrix}
0 & \mu(  d)  \,\\
0 & 0
\end{bmatrix}
\in\mathscr{B}(H^{1/2}(\Gamma)\oplus H^{-1/2}(\Gamma),H^{-1/2}(\Gamma)\oplus
H^{1/2}(\Gamma))\right.
\end{align*}
and consider the operator-valued function
\begin{equation}
\kappa\mapsto(1+(B_{1}(cd)\oplus B_{2}(d\,))\tau G_{\kappa})\in\mathscr{B}%
(L^{2}(\mathbb{R}^{3})\oplus H^{-1/2}(\Gamma)\oplus H^{1/2}(\Gamma))\,,
\label{Q_Krein}%
\end{equation}
where%
\[
\tau G_{\kappa}:=%
\begin{bmatrix}
R_{\kappa} & S\!L_{\kappa} & D\!L_{\kappa}\\
\gamma_{0}R_{\kappa} & S_{\kappa} & K_{\kappa}^{\ast}\\
\gamma_{1}R_{\kappa} & K_{\kappa} & D_{\kappa}%
\end{bmatrix}
\,.
\]
Then (\ref{L_cd_weak}) recasts as
\begin{equation}
L_{c,d}=-\Delta_{\mathbb{R}^{3}}-B_{1}(cd)+(\gamma_{0}\oplus\gamma_{1})^{\ast
}B_{2}(d\,)(\gamma_{0}\oplus\gamma_{1})\,. \label{A_cd_weak_id}%
\end{equation}
Our aim is to prove, through a Kre\u{\i}n-type resolvent formula, the
self-adjointness of $L_{c,d}$.

\begin{theorem}
\label{Theorem_Krein_EM}$L_{c,d}$ is self-adjoint in $L^{2}(\mathbb{R}%
^{3},c^{-1})$ and the resolvent formula%
\begin{equation}
R_{\kappa}(  c,d)  :=\left(  L_{c,d}-\kappa^{2}\right)
^{-1}=R_{\kappa}+%
\begin{bmatrix}
R_{\kappa}1_{\Omega}^{\ast} & S\!L_{\kappa}%
\end{bmatrix}
Q_{\kappa}^{-1}(  c,d)
\begin{bmatrix}
(cd-1)\Delta_{\Omega}^{\mathrm{max}}1_{\Omega}R_{\kappa}\\
\mu(  d)  \,\gamma_{1}R_{\kappa}%
\end{bmatrix}
\label{Krein_EM}%
\end{equation}
holds for any $\kappa\in\mathbb C_{+}$ such that $\kappa^{2}\in{\varrho}(L_{c,d})$.
\end{theorem}
\begin{proof}
Most of the results from \cite[Sec. 4]{MaPo Acoustic 26} apply our setting
after replacing the coefficients $\rho_{\ell}$ (there denoting the acoustic
densities of $\Omega_{\ell}$) with $d_{\ell}^{-1}$ (c.f. the definitions
(\ref{c_def}) with \cite[eq. (1.2)]{MaPo Acoustic 26}). In particular, from
\cite[eqs. (4.10) and (4.11) ]{MaPo Acoustic 26} there follows that
(\ref{Q_Krein}) is invertible if and only if $Q_{\kappa}^{-1}(c,d)$ exists and
in such a case one has (c.f. with \cite[proof of Theorem 4.6]{MaPo Acoustic
26})\emph{ }%
\begin{align*}
&G_{\kappa}(1+(B_{1}(cd)\oplus B_{2}(d\,))\tau G_{\kappa})^{-1}(B_{1}(cd)\oplus
B_{2}(d\,))G_{-\bar{\kappa}}^{\ast}\\
=&%
\begin{bmatrix}
R_{\kappa}1_{\Omega}^{\ast} & S\!L_{\kappa}%
\end{bmatrix}
Q_{\kappa}^{-1}(c,d)%
\begin{bmatrix}
(cd-1)1_{\Omega}^{\ast}\Delta_{\Omega}^{\mathrm{max}}1_{\Omega}R_{\kappa}\\
\mu(  d)  \,\gamma_{1}R_{\kappa}%
\end{bmatrix}
.
\end{align*}
Then, by \cite[Theorem 4.2]{MaPo Acoustic 26}, if $Q_{\kappa}^{-1}(c,d)$ for
some $\kappa$ then there exists a closed operator $\hat{L}$ with resolvent
given by the right hand side of \eqref{Krein_EM}; such a resolvent exists (by
the same arguments as in \cite[Theorem 2.19 and Remark 2.20]{CFP}) for any
$\kappa\in{\mathbb{C}}_{+}$ such that $\kappa^{2}\in{\varrho}(\hat{L}%
)\cap\mathbb{C}\backslash\lbrack0,+\infty)$. Since, by Theorem
\ref{Theorem_Q_full}, $Q_{\kappa}^{-1}(c,d)$ exists for $\kappa\in\mathbb{C}$
in an open set, $\widehat{L}$ exists, his resolvent is provided by the r.h.s.
of (\ref{Krein_EM}) and, by \cite[Theorem 4.3]{MaPo Acoustic 26}, it coincides
with $L_{c,d}$. Furhtermore, since $L_{c,d}$ is symmetric, $\pm i\in{\varrho
}(L_{c,d})$ and so $\mathrm{ran}(L_{c,d}\pm i)=\mathrm{dom}((L_{c,d}\pm
i)^{-1})=L^{2}(\mathbb{R}^{3},c^{-1})$. In conclusion, $L(c,d)$ is self-adjoint.
\end{proof}

\begin{remark}
By
\[
\left(  (-\Delta_{\Omega_{\mathrm{ex}}}^{\mathrm{max}%
})\oplus(-cd\Delta_{\Omega}^{\mathrm{max}})-\kappa^{2}\right)  (L_{c,d}%
-\kappa^{2})^{-1}=1
\]
and by%

\[
\left(  (-\Delta_{\Omega_{\mathrm{ex}}}^{\mathrm{max}%
})\oplus(-cd\Delta_{\Omega}^{\mathrm{max}})-\kappa^{2}\right)  R_{\kappa
}u=\left(  \chi_{\Omega_{\mathrm{ex}}}+cd\chi_{\Omega
}\right)  u-(1-cd)\kappa^{2}\chi_{\Omega}R_{\kappa}u\,,
\]
there follows
\begin{align}
&  \left(  (-\Delta_{\Omega_{\mathrm{ex}}}^{\mathrm{max}%
})\oplus(-cd\Delta_{\Omega}^{\mathrm{max}})-\kappa^{2}\right)
\begin{bmatrix}
R_{\kappa}1_{\Omega}^{\ast} & S\!L_{\kappa}%
\end{bmatrix}
Q_{\kappa}^{-1}(  c,d)
\begin{bmatrix}
(cd-1)\Delta_{\Omega}^{\mathrm{max}}1_{\Omega}R_{\kappa}u\\
\mu(  d)  \,\gamma_{1}R_{\kappa}u
\end{bmatrix}
\nonumber\\
=  &  u-\left(  \chi_{\Omega_{\mathrm{ex}}}%
+cd\chi_{\Omega}\right)  u-\left(  1-cd\right)  \kappa^{2}\chi_{\Omega
}R_{\kappa}u=\left(  1-cd\right)  \chi_{\Omega}\left(  1-\kappa^{2}R_{\kappa
}\right)  u\,. \label{Krein_diff_id}%
\end{align}

\end{remark}


We conclude this Section by providing a Limiting Absorption Principe (LAP for short) for $R_{\kappa}(c,d)$. 
At first, let us recall LAP for the free resolvent: 
the limits%
\[
\mathcal R_{\lambda}^{\pm}:=\lim_{\delta\rightarrow0^{+}}(-\Delta-(\lambda\pm i\delta))^{-1}\,,\qquad \lambda\in\mathbb R\,,
\]
exist in ${\mathscr B}(  H_{\alpha }^{0}(  \mathbb{R}^{3})  ,H_{-\alpha 
}^{2}(  \mathbb{R}^{3}) )$ 
for $\alpha >3/2$ and one has%
\[
(  -\Delta+\lambda)  \mathcal R_{\lambda}^{\pm}u=u\,,\qquad u\in H_{\alpha }%
^{0}(  \mathbb{R}^{3})  \,.
\]
By the obvious relation
$$
(-\Delta-\kappa^{2})^{-1}=
\big(-\Delta-\text{Re}(\kappa)^{2}+\text{Im}(\kappa)^{2}-i2\text{Re}(\kappa)\text{Im}(\kappa)\big)^{-1}\,,\qquad \kappa\in\mathbb C_{+}\,,
$$
one gets that the  limit 
$$
R_{\kappa}:=\lim_{\delta\rightarrow0^{+}}R_{\kappa+i\delta}\,,\qquad \kappa\in\mathbb R\,,
$$
exists in ${\mathscr B}(  H_{\alpha }^{0}(  \mathbb{R}^{3})  ,H_{-\alpha 
}^{2}(  \mathbb{R}^{3}) )$ 
for $\alpha >3/2$ and one has%
\[
(  -\Delta-\kappa^{2})   R_{\kappa}u=u\,,\qquad u\in H_{\alpha }%
^{0}(  \mathbb{R}^{3})  \,,\qquad\kappa\in\mathbb{R}\,.
\]
By duality and interpolation, the  mapping properties of $R_{\kappa}$ extend according
to%
\begin{equation}
R_{\kappa}\in{\mathscr B}(  H_{\alpha }^{s-2}(  \mathbb{R}^{3})
,H_{-\alpha }^{s}(  \mathbb{R}^{3})  )  \,,\qquad\lambda
\in\mathbb{R}\,,\ \alpha >3/2\,,\ s\in\left[  0,2\right]  \,. \label{LAP}%
\end{equation}
Moreover, the layer potentials inherit these limit properties; in particular, one has%
\begin{equation}
S\!L_{\kappa}\in{\mathscr B}(  H^{s-2}(  \Gamma)  ,H_{-\alpha 
}^{s-1/2}(  \mathbb{R}^{3})  )  \,,\qquad\kappa\in
\mathbb{R}\,,\ \alpha >3/2\,,\ s\in\left[  0,2\right)  \,, \label{SL_LAP}%
\end{equation}
see e.g. in \cite[Sec. 5.5.1]{MaPo JMPA 19}.

\begin{lemma}
\label{Lemma_LAP} The limit%
\begin{equation}
R_{\kappa}(  c,d)  :=\lim_{\delta\rightarrow0^{+}%
}R_{\kappa+i\delta}(  c,d)  \,, \qquad \kappa\in\mathbb R\backslash \{\kappa\in\mathbb R:\ker(Q_{\kappa}(c,d))\not=\{0\}\}\,,\label{LAP_L_cd}%
\end{equation}
exists in ${\mathscr B}(  H_{\alpha }%
^{0}(  \mathbb{R}^{3})  ,H_{-\alpha }^{1}(  \mathbb{R}^{3}))$  for $\alpha >3/2$. The resolvent formula (\ref{Krein_EM}) holds in
the limit and, for any  $u\in H_{\alpha }^{0}(  \mathbb{R}^{3})$ one has%
\begin{equation}
\big(  (  -\Delta_{\Omega_{\+}})  \oplus(
-cd\Delta_{\Omega})  -\kappa^{2}\big)  R_{\kappa}(  c,d)
u=u\,,\qquad   
\kappa\in\mathbb R\backslash \{\kappa\in\mathbb R:\ker(Q_{\kappa}(c,d))\not=\{0\}\}
\,. \label{LAP_L_cd_id}%
\end{equation}

\end{lemma}

\begin{proof}
By \eqref{Krein_EM}), for any $\kappa\in\mathbb R$ and $\delta>0$, there holds
\[
R_{\kappa+i\delta}(  c,d)  =R_{\kappa+i\delta }%
+\begin{bmatrix}R_{\kappa+i\delta}1_{\Omega}^{*}&S\!L_{\kappa+i\delta}\end{bmatrix}
Q_{\kappa+i\delta}^{-1}(  c,d)
\begin{bmatrix}(  cd-1)1_{\Omega
}^{\ast}\Delta_{\Omega}^{\mathrm{max}}1_{\Omega}R_{\kappa+i\delta}
\\
\tilde{d}\,\gamma_{1} R_{\kappa+i\delta}
\end{bmatrix}.
\]
By (\ref{LAP})-(\ref{SL_LAP}), the r.h.s. of this equation has limit in
${\mathscr B}\left(  H_{\alpha }^{0}(  \mathbb{R}^{3})  ,H_{-\alpha }%
^{1}(  \mathbb{R}^{3})  \right)  $ with the possible exception of
the real poles of $\kappa\mapsto  Q_{\kappa}^{-1}(  c,d)  $, i.e., for the real $\kappa$'s such that $\ker(Q_{\kappa}(c,d))\not=\{0\}$. Hence, the limit  in
(\ref{LAP_L_cd}) is  defined through (\ref{Krein_EM}) and the relation
(\ref{LAP_L_cd_id}) follows by taking the limit in 
(\ref{Krein_diff_id}).
\end{proof}

\section{Spectral analysis}

\begin{theorem}
\label{theorem_spectrum}$\left.  {}\right.  \medskip$\newline$i)$ $\sigma
_{p}(L_{c,d})\cap\left(  0,+\infty\right)  =\varnothing\,$.\medskip
\newline$ii)$ $\sigma_{ess}(L_{c,d})=[0,+\infty)$ and $\sigma_{disc}(L_{c,d})$
accumulates at $-\infty\,$.\medskip\newline$iii)$ $\sigma_{sc}(L_{c,d}%
)=\varnothing$ and $\sigma_{ac}(L_{c,d})=\left[  0,+\infty\right)  \,$.
\end{theorem}

\begin{proof}
$i)$ Let $\lambda>0$ and $u_{\lambda}\in\mathcal{D}_{d}$ be such that
$L_{c,d}u_{\lambda}=\lambda u_{\lambda}$. Since $L_{c,d}\subset(-\Delta
_{\Omega_{\mathrm{ex}}}^{\mathrm{max}})\oplus
(-cd\Delta_{\Omega}^{\mathrm{max}})$, it results
\[
(\Delta+\lambda)u_{\lambda}|\Omega_{\mathrm{ex}}=0\,.
\]
By Rellich's estimate (see, e.g., \cite[Corollary 4.8]{Leis}): $\left.
u_{\lambda}\right\vert B^{c}=0$ for any open ball $B\supset\overline{\Omega}$.
Hence, by the unique continuation property (see, e.g., \cite[Section
4.3]{Leis}), $\left.  u_{\lambda}\right\vert \mathbb{R}^{3}\backslash
\overline{\Omega}=0$ and so $\gamma_{0}^{+}u_{\lambda}=\gamma_{1}%
^{+}u_{\lambda}=0$. Hence, the boundary conditions in $\mathcal{D}_{d}$ imply%
\[
\left\{
\begin{array}
[c]{l}%
(cd\,\Delta+\lambda)u_{\lambda}|\Omega=0\,,\\
\gamma_{0}^{\mathrm{in}}=\gamma_{1}^{\mathrm{in}}u=0\,,
\end{array}
\right.
\]
which admits only the trivial solution. Therefore, $u_{\lambda}=0$.

$ii)$ Since $S\!L_{\kappa}\in\mathscr{B}(H^{-3/2}(\Gamma),L^{2}(\mathbb{R}%
^{3}))$ and $H^{-1/2}(\Gamma)\hookrightarrow H^{-3/2}(\Gamma)$ with compact
embedding, one has $S\!L_{\kappa}\in\mathfrak{S}_{\infty}(H^{-1/2}%
(\Gamma),L^{2}(\mathbb{R}^{3}))$. By $R_{\kappa}\in\mathscr{B}\left(
L^{2}(  \mathbb{R}^{3})  ,H^{2}(  \mathbb{R}^{3})
\right)  $ and by the compact embedding $H^{2}(\Omega)\hookrightarrow
L^{2}(\Omega)$, one has $1_{\Omega}R_{\kappa}\in\mathfrak{S}_{\infty}%
(L^{2}(\mathbb{R}^{3}),L^{2}(\Omega))$; then, by duality, $R_{\kappa}%
1_{\Omega}^{\ast}\in\mathfrak{S}_{\infty}(L^{2}(\Omega),L^{2}(\mathbb{R}%
^{3}))$. Hence, by \eqref{Krein_EM}, the resolvent difference $\left(
L_{c,d}-\kappa^{2}\right)  ^{-1}-R_{\kappa}$, is compact as well. Therefore,
the Weyl theorem applies (see, e.g., \cite[Th. XIII.14]{ReSi IV}) and
$\sigma_{ess}(L_{c,d})=\sigma_{ess}(-\Delta)=[0,+\infty)$. Since $L_{c,d}$ is
unbounded from below and $\sigma_{disc}(L_{c,d})\subset\mathbb{R}%
\backslash\sigma_{ess}(L_{c,d})=(-\infty,0)$, it accumulates at $-\infty$.

$iii)$ Since $E:=\{\kappa\in \mathbb R:
\ker(Q_{\kappa}(c,d))\not=\{0\}\}$ is discrete, $E^{2}:=\{\kappa^{2}:\kappa\in E\}\subset
[0,+\infty)$ is discrete as well and $[0,+\infty)\backslash E^{2}$ is the (almost countable) union of open intervals $I_{n}$. Now, we argue as in 
 \cite[Proof of Theorem 6.1]{Agmon}. Hence, we consider the map 
$I_{n}\ni\lambda\mapsto\left\langle f,P_{\lambda}
f\right\rangle $,, where  
$f\in L^{2}(\mathbb R^{3},c^{-1})_{pp}^{\bot}$,  $P_{\lambda
}    $ is the spectral resolution of the identity
associated with $L_{c,d}$ and $  
L^{2}(\mathbb R^{3}\!,c^{-1}) _{pp}$ denotes the pure point subspace of $L_{c,d}$.
Given the compact interval $[\lambda_{1},\lambda_{2}]\subset I_{n}$ and $f\in L^{2}_{\alpha}
(\mathbb{R}^{3})$, by Lemma \ref{Lemma_LAP}
and by Stone's formula one has
\begin{equation*}
\big\langle f,(P_{\lambda_{2}}-P_{\lambda_{1}})f\big\rangle_{L^{2}(\mathbb R^{3})} =%
\frac{1}{2\pi i}\int_{\lambda_{1}}^{\lambda_{2}}\big\langle f,(R_{\lambda^{1/2} }(c,d) -R_{-\lambda^{1/2} }(c,d)
)f\big\rangle_{\alpha,-\alpha}\, d\lambda \,,
\end{equation*}
where $\langle\cdot,\cdot\rangle_{\alpha,-\alpha}$ denotes the $L^{2}_{\alpha}
(\mathbb{R}^{3})$-$L^{2}_{-\alpha}
(\mathbb{R}^{3})$ duality.
Hence, $\lambda\mapsto\langle f,P_{\lambda} f\rangle_{L^{2}(\mathbb R^{3})} $ is continuously differentiable for any $f\in L^{2}_{\alpha}(  \mathbb{R}^{3})
$. By the same argument as in \cite[Proof of Theorem 6.1]{Agmon}, the property
extends to the whole $(    (  L^{2}(  \mathbb{R}^{3}))  ^{pp})    ^{\bot}$ and so $I_{n}\subseteq \sigma_{ac}(L_{c,d})$. This shows that $\sigma_{sc}(  L_{c,d}) \cap[0,+\infty)\subseteq E_{2}$, which implies the thesis, since a singular continuous measure can't have atoms.

Therefore
\[
\sigma_{disc}(L_{c,d})\cup\sigma_{ess}(L_{c,d})=\sigma(L_{c,d})=\overline
{\sigma_{p}(L_{c,d})}\cup\sigma_{ac}(L_{c,d})\cup\sigma_{sc}(L_{c,d}%
)=\overline{\sigma_{p}(L_{c,d})}\cup\sigma_{ac}(L_{c,d})
\]
and since $\sigma_{p}(L_{c,d})\backslash\sigma_{disc}(L_{c,d})\subseteq\{0\}$,
one has either $\sigma_{ac}(L_{c,d})=\sigma_{ess}(L_{c,d})$ or $\{0\}\cup
\sigma_{ac}(L_{c,d})=\sigma_{ess}(L_{c,d})$. Since $\sigma_{ac}(L_{c,d})$ is
closed, one gets $\sigma_{ac}(L_{c,d})=[0,+\infty)$.
\end{proof}
\par\smallskip\smallskip
Next we analyze the spectral nature of $\kappa=0$.

\begin{theorem}
\label{Theorem_zero_eigenvalue} If $d_{\ell}=d_{\circ}\in(-\infty,0)\backslash\{  -1\}$
 for any index $\ell$ and $\mu(  d_{\circ})^{-1}  \in\sigma(K_{0})$, then $0\in\sigma_{p}(L_{c,d})$.
\end{theorem}
\begin{proof} For brevity, we use the notation $\lambda_{\circ}\equiv\mu(
d_{\circ}) ^{-1}\in\sigma(  K_{0}) $. Given $\varphi_{\circ}\in\ker(  \lambda_{\circ}- K_{0})\backslash\{0\}  $, let $u_{0}:=S\!L_{0}\varphi_{\circ}$. One has $\Delta_{\Omega}(u_{0}|\Omega)=0$, $\Delta_{\Omega_{\+}}(u_{0}|\Omega_{\+})=0$ and, by the jump properties of the single layers, 
$$
[  \gamma_{0}]  u_{0}=0\,,
\qquad
[  \gamma_{1}]  u_{0}=-\varphi_{\circ}=-\lambda_{\circ}^{-1}K_{0}\varphi_{\circ}=-\lambda_{\circ}^{-1}\gamma_{1}S\!L_{0}\varphi_{\circ}=-\lambda_{\circ}^{-1}u_{0}\,.
$$
Therefore, by (\ref{L_cd_def}) and (\ref{transmission_bc}),  if  $u_{0}\in H^{1}(\mathbb R^{3})$ then $u_{0}\in\dom(L_{c,d})$ and $L_{c,d}u_{0}=0$  (hence, $0\in\sigma_{p}(L_{c,d})$) .
\par
Since $\varphi_{\circ}\in H^{-1/2}(  \Gamma)  $,  one has
$\chi_{B_{r}}S\!L_{0}\varphi_{\circ}\in H^{1}(  \mathbb{R}^{3})$ for any $r>0$; to conclude we need
to prove that $\chi_{\mathbb{R}^{3}\backslash B_{r}}S\!L_{0}\varphi_{\circ}\in H^{1}(  \mathbb{R}^{3})  $ for $r$ large enough. In the following we take $r>0 $ so that $\Omega\subset\subset B_{r}$. Defining $\hat x:={x}/{|x|}$, one has
$$
\frac{1}{|x-y| }=
\frac{1}{| x|}\,\frac{1}{| \hat{x}-y/| x|\,| }=\frac{1}{| x|}\left(1+\int_{0}^{1/|x|}\frac{(\hat x-sy)\cdot y}{|\hat x-sy|^{3}}\,ds\right)=
\frac{1}{| x|}+\frac{h(x,y)}{| x|^{2}}\,,
$$
$$
h(x,y):=\int_{0}^{1}\frac{(\hat x-sy/|x|\,)\cdot y}{|\hat x-sy/|x|\,|^{3}}\,ds\,.
$$
Similarly, defining $\widehat{x-y  }:={(x-y)  }/{| x-y|}$, one has
\begin{align*}
\nabla_{\!x}\,\frac{1}{| x-y| }=&\frac{\widehat{x-y  }}{| x-y| ^{2}}=\frac{\widehat{
x-y  }}{|x| ^{2}}\,\frac{1}{| \hat x-y/|x|\, | ^{2}}=\frac{\widehat{
x-y  }}{|x| ^{2}}\left(1+2\int_{0}^{1/|x|}\frac{(\hat x-sy)\cdot y}{|\hat x-sy|^{4}}\,ds\right)\\
=&\left(
\frac{1}{|x|^{2}}+\frac{\tilde h(x,y)}{|x|^{3}}\right)\widehat{x-y  }\,,
\end{align*}
$$
\tilde h(x,y):=2\int_{0}^{1}\frac{(\hat x-sy/|x|\,)\cdot y}{|\hat x-sy/|x|\,|^{4}}\,ds\,.
$$
This entails, for any $\varphi$ sufficiently regular,
\begin{align*}
\left(  \chi_{\mathbb{R}^{3}\backslash B_{r}}S\!L_{0}\varphi\right)  (
x)  
= \frac{\chi_{\mathbb{R}^{3}\backslash B_{r}}(  x)  }{|x| }\int_{\Gamma}\varphi(  y)\,d\sigma_{\Gamma}(y)
+\frac{\chi_{\mathbb{R}^{3}\backslash B_{r}}(  x)  }{|x|^{2} }\int_{\Gamma}h(x,y)
\varphi(  y)\,d\sigma_{\Gamma}(y)
\end{align*}
and
\begin{align*}
\left(  \chi_{\mathbb{R}^{3}\backslash B_{r}}\nabla S\!L_{0}\varphi\right)  (
x)  =\frac{\chi_{\mathbb{R}^{3}\backslash B_{r}}(  x)  }{|x|^{2} }\int_{\Gamma}(\widehat{x-y  })\varphi(  y)\,d\sigma_{\Gamma}(y)
+\frac{\chi_{\mathbb{R}^{3}\backslash B_{r}}(  x)  }{|x|^{3} }\int_{\Gamma}\tilde h(x,y)(\widehat{x-y  })\varphi(  y)\,d\sigma_{\Gamma}(y)\,.
\end{align*}
Therefore (notice that $\widehat{x-(\cdot)}$, $h(x,\cdot)$ and $\tilde h(x,\cdot)$ belong to $\mathcal C^{1}(\Gamma)$ whenever $x$ is far away from $\Omega$), 
\begin{align*}
|\left(  \chi_{\mathbb{R}^{3}\backslash B_{r}}S\!L_{0}\varphi_{\circ}\right)  (
x)|^{2}\lesssim &\frac{\chi_{\mathbb{R}^{3}\backslash B_{r}}(  x)  }{|x|^{2} }\ \big|\langle1,\varphi_{\circ}\rangle_{\frac12,-\frac12}\big|^{2}+\frac{\chi_{\mathbb{R}^{3}\backslash B_{r}}(  x)  }{|x|^{4} }\, \sup_{|x|>r}\|h(x,\cdot)\|_{\frac12}\|\varphi_{\circ}\|_{-\frac12}
\\
\lesssim &\frac{\chi_{\mathbb{R}^{3}\backslash B_{r}}(  x)  }{|x|^{2} }\ \big|\langle1,\varphi_{\circ}\rangle_{\frac12,-\frac12}\big|^{2}+\frac{\chi_{\mathbb{R}^{3}\backslash B_{r}}(  x)  }{|x|^{4} }\, \sup_{|x|>r}\|h(x,\cdot)\|_{1}\|\varphi_{\circ}\|_{-\frac12}
\end{align*}
and
\begin{align*}
&|\left(  \chi_{\mathbb{R}^{3}\backslash B_{r}}\nabla S\!L_{0}\varphi_{\circ}\right)  (
x)|^{2}\\
\lesssim &\frac{\chi_{\mathbb{R}^{3}\backslash B_{r}}(  x)  }{|x|^{4} }\ 
\, \sup_{|x|>r}\|\widehat{x-(\cdot)}\|_{\frac12}\|\varphi_{\circ}\|_{-\frac12}
+\frac{\chi_{\mathbb{R}^{3}\backslash B_{r}}(  x)  }{|x|^{6} }\,\sup_{|x|>r}\|\widehat{x-(\cdot)}\tilde h(x,\cdot)\|_{\frac12}\|\varphi_{\circ}\|_{-\frac12}
\\
\lesssim &\frac{\chi_{\mathbb{R}^{3}\backslash B_{r}}(  x)  }{|x|^{4} }\ 
\, \sup_{|x|>r}\|\widehat{x-(\cdot)}\|_{1}\|\varphi_{\circ}\|_{-\frac12}
+\frac{\chi_{\mathbb{R}^{3}\backslash B_{r}}(  x)  }{|x|^{6} }\,\sup_{|x|>r}\|\widehat{x-(\cdot)}\tilde h(x,\cdot)\|_{1}\|\varphi_{\circ}\|_{-\frac12}\,.
\end{align*}
Let us recall that 
\begin{equation}\label{due}
\langle 1,\varphi\rangle_{\frac12,-\frac12}=[S_{0}^{-1}1,\varphi]_{-\frac12}\,,
\end{equation}
where $[\cdot,\cdot]_{-\frac12}$ denotes the scalar product in $H^{-\frac12}(\Gamma)$ with respect to which $K_{0}$ is self-adjoint and  that $S_{0}^{-1}1$ is the eigenvector of $K_{0}$ corresponding to the eigenvalue $-\frac12$.  
Since $d_{\circ}\not=-1$,  for the eigenvalue $\lambda_{\circ}$ there holds $\lambda_{\circ}=-\frac12 $ and so the eigenfunction $\varphi_{\circ}$ is always orthogonal to $S_{0}^{-1}1$. Moreover, 
$$
\sup_{| x|>r}\|h(x,\cdot)\|_{H^{1}(\Gamma)}<+\infty\,,\quad \sup_{|x|>r}\|\widehat{x-(\cdot)}\|_{H^{1}(\Gamma)}<+\infty\,,\quad \sup_{|x|>r}\|\widehat{x-(\cdot)}\tilde h(x,\cdot)\|_{H^{1}(\Gamma)}<+\infty\,.
$$
Therefore,
$$
\chi_{\mathbb{R}^{3}\backslash B_{r}}(x)\left(|\left(  S\!L_{0}\varphi_{\circ}\right)  (
x)|^{2}+|\left(  \nabla S\!L_{0}\varphi_{\circ}\right)  (
x)|^{2}\right)\lesssim \frac{\chi_{\mathbb{R}^{3}\backslash B_{r}}(  x)  }{|x|^{4} }
$$
and the proof is done.
\end{proof}
\section{Resonances}
By (\ref{L_cd_def}), $L_{c,d}$ coincides with $-\Delta$ when restricted to
functions with support outside a compact set containing $\Omega$; hence, the
study of the resonances of $L_{c,d}$ fits the blackbox approach \cite{DyaZwo}. This entails the following definition of the set of resonances of $L_{c,d}$, denoted by $\mathrm{Res}(L_{c,d})$: 
\begin{definition}
$\kappa^{2}_{\circ}\in\mathrm{Res}(L_{c,d})$ whenever $\kappa_{\circ}\in\mathbb{C}_{-}$ and $\kappa\mapsto1_{U_{2}%
}R_{\kappa}(  c,d)  1_{U_{1}}^{\ast}$ has a pole in $\kappa_{\circ}$ for some bounded open sets $U_{1}$ and $U_{2}$ containing
$\Omega$.
\end{definition}

\begin{lemma}
Let $U_{1}$ and $U_{2}$ be two bounded open sets containing $\Omega$. Then, 
\begin{equation}
\mathbb C\ni\kappa\mapsto 1_{U_{2}}R_{\kappa}(c,d)1_{U_{1}}^{\ast}=1_{U_{2}}R_{\kappa}1_{U_{1}}^{\ast}+%
\begin{bmatrix}
1_{U_{2}}R_{\kappa}1_{\Omega}^{\ast} & 1_{U_{2}}S\!L_{\kappa}%
\end{bmatrix}
Q_{\kappa}^{-1}(  c,d)
\begin{bmatrix}
(cd-1)\Delta_{\Omega}^{\mathrm{max}}1_{\Omega}R_{\kappa}1_{U_{1}}^{\ast}\\
\mu(  d)  \,\gamma_{1}R_{\kappa}1_{U_{1}}^{\ast}%
\end{bmatrix}
\,. \label{Krein_EM_cut}%
\end{equation}
is a $\mathscr{B}(L^{2}(U_{1}),H^{1}(U_{2}%
))$-valued meromorphic map.
\end{lemma}

\begin{proof}
Let $\psi\in\mathcal{C}_{c}^{\infty}(\mathbb{R}^{3})$ such that $\psi
|U_{1}\cup U_{2}=1$. Then, by $S\!L_{k}=R_{k}\gamma_{0}^{\ast}$, by
$\gamma_{0}^{\ast}\in\mathscr{B}(H^{-1/2}(\Gamma),H^{-1}(\mathbb{R}^{3}))$,
supp$(\gamma_{0}^{\ast}\phi)=\Gamma$, and by Lemma \ref{Lemma_R_k_analytic},
\[
\mathbb{C}\ni\kappa\mapsto%
\begin{bmatrix}
1_{U_{2}}R_{\kappa}1_{\Omega}^{\ast} & 1_{U_{2}}S\!L_{\kappa}%
\end{bmatrix}
\]
is a $\mathscr{B}(H^{-1/2}(\Gamma),,H^{1}(U_{2}))$-valued analytic map.
Similarly,
\[
\mathbb{C}\ni\kappa\mapsto%
\begin{bmatrix}
(cd-1)\Delta_{\Omega}^{\mathrm{max}}1_{\Omega}R_{\kappa}1_{U_{1}}^{\ast}\\
\mu(  d)  \,\gamma_{1}R_{\kappa}1_{U_{1}}^{\ast}%
\end{bmatrix}
\]
is a $\mathscr{B}(L^{2}(U_{1}),H^{-1/2}(\Gamma))$-valued analytic map. Hence,
by Theorem \ref{Theorem_Q_full}, the r.h.s. of (\ref{Krein_EM_cut}) is a
$\mathscr{B}(L^{2}(U_{1}),H^{1}(U_{2}))$-valued meromorphic map. This coincides
with the l.h.s. of (\ref{Krein_EM_cut}) by (\ref{Krein_EM}) and by unique
meromorphic continuation principle.
\end{proof}
\begin{theorem}\label{teo-res}
$$
\mathrm{Res}(L_{c,d})=\{\kappa^{2}\in\mathbb C:\ker(  Q_{\kappa}(
c,d))  \neq\{  0\}\,,\ \kappa\in\mathbb C_{-}\}\,. 
$$
\end{theorem}
\begin{proof}
Let 
$$ \kappa\mapsto\Lambda_{\kappa}:=%
\begin{bmatrix}
1_{U_{2}}R_{\kappa}1_{\Omega}^{\ast} & 1_{U_{2}}S\!L_{\kappa}%
\end{bmatrix}
\,,\qquad
 \kappa\mapsto\Theta_{\kappa}:=%
\begin{bmatrix}
(cd-1)\Delta_{\Omega}^{\mathrm{max}}1_{\Omega}R_{\kappa}1_{U_{1}}^{\ast}\\
\mu(  d)  \,\gamma_{1}R_{\kappa}1_{U_{1}}^{\ast}%
\end{bmatrix}
$$
be the analytic, operators-valued maps on the left and on the right of $Q_{\kappa}^{-1}(  c,d)
$ in \eqref{Krein_EM_cut}. $\Theta_{\kappa}$ has dense range in $L^{2}(\Omega)\oplus
H^{-1/2}(\Gamma)$ and $\Lambda_{\kappa}$ is injective. Therefore, the poles in $
\mathbb{C}_{-} $ of
the map in \eqref{Krein_EM_cut} coincides with the poles in $\mathbb{C}\backslash
\mathbb{C}_{+} $ of $\kappa\mapsto
Q_{\kappa}^{-1}(c,d)$, i.e., with the points in the set $\{\kappa\in
\mathbb{C}_{-}:\ker(  Q_{\kappa}(
c,d))  \neq\{  0\}\}$. Otherwise, denoting by ${\ran}(
\Theta_{\kappa})  _{\Vert}$ the orthogonal projection onto the closed
subspace $\ker(Q_{\kappa}(c,d))$, we would have ${\ran}(\Theta_{\kappa
})_{\Vert}\cap\ker(Q_{\kappa}(c,d))=\{0\}$ which would imply ${\ran}%
(\Theta_{\kappa})\subseteq\left(  \ker Q_{\kappa}(c,d)\right)  ^{\perp}$, in
contradiction with the density of ${\ran}(\Theta_{\kappa})$.
\end{proof}
\begin{remark}\label{rem-pole}
By the same kind of reasonings as in the proof of Theorem \ref{teo-res}, one gets that if $\kappa_{\circ}\in\mathbb C\backslash\mathbb C_{-}$ is such that $\ker(Q_{\kappa_{\circ}}(c,d))\not=\{0\}$ then   $\kappa\mapsto R_{\kappa}(
c,d)$ has a pole at $\kappa_{\circ}$; hence, $\kappa_{\circ}^{2}$ is an isolated eigenvalue of $L_{c,d}$.
\end{remark}
\begin{remark}\label{rem-res} Notice that there are no $\kappa\in\mathbb R\backslash\{0\}$ such that $\ker(  Q_{\kappa}(
c,d))  \neq\{  0\}$. Indeed, such poles of $\kappa\mapsto Q_{\kappa}(
c,d)$ should correspond to strictly positive isolated eigenvalues of $L_{c,d}$. This is forbidden by $\sigma_{p}(L_{cd})\cap(0,+\infty)=\varnothing$ (see Theorem \ref{theorem_spectrum}). Furthermore, the poles in $\mathbb C_{+}$ of $\kappa\mapsto Q_{\kappa_{\circ}}(c,d)$  correspond to the isolated negative eigenvalues of $L_{c,d}$.
Hence, if $\kappa_{\circ}\in\{\mathbb C\backslash\{0\}:\ker(Q_{\kappa}(c,d))\not=\{0\}\}$ then $\kappa^{2}_{\circ}$ is either a resonance 
or a negative eigenvalue.
\end{remark}
\begin{remark}\label{rem-LAP} By Remark \ref{rem-res}, the limit \eqref{LAP_L_cd} and the relation \eqref{LAP_L_cd_id} hold for any $\kappa\in\mathbb R\backslash\{0\}$. 
\end{remark}

\section{Near-zero eigenvalues and resonances}

In this Section we consider the case $n=1$ and avoid the use of the index
$\ell$; furthermore, we parametrize $L_{c,d}$ by a small $\varepsilon>0$
asymptotically close to the distance of $\mu(  d)^{-1}  $ from
$\sigma(N\!P_{0})=\sigma(K_{0})$.

\begin{description}
\item[\emph{H1}] $\Omega\subset\mathbb{R}^{3}$ is open, simply connected,
bounded with $\mathcal{C}^{2}$ boundary $\Gamma$. We take $c=c(\varepsilon)$
and $d=d(\varepsilon)$ such that%
\begin{equation}
c(\varepsilon)=c_{0}+c_{1}\varepsilon+\mathcal{O}(\varepsilon^{2})\,,\qquad
d(\varepsilon)=\frac{\lambda_{\circ}-1/2-\rho(\varepsilon)
}{\lambda_{\circ}+1/2-\rho(\varepsilon)}\,,\qquad\rho(\varepsilon)=\rho_{0}\varepsilon
+\mathcal{O}(\varepsilon^{2})\,, \label{eps_fun}%
\end{equation}
where $c_{0}>0$, $\rho_{0}\neq0$ and $c_{1}\neq0$. 

\item[\emph{H2}] $\lambda_{\circ}\in\sigma(K_{0})\backslash\left\{
0,-1/2\right\}  $ is simple; $\varphi_{\circ}\in H^{-1/2}(\Gamma)$ denotes the
corresponding normalized eigenvector.
\end{description}

Denoting%
\[
\tilde{d}(  \varepsilon)  \equiv\mu\left(  d(\varepsilon)\right)
=2\,\frac{1-d(\varepsilon)}{1+d(\varepsilon)}%
\]
one has%
\begin{equation}
\tilde{d}^{-1}(  \varepsilon)  =\lambda_{\circ}-\rho\left(
\varepsilon\right)  \,. \label{d_tilda_eps}%
\end{equation}
In particular, for $\varepsilon$ small enough, $\tilde{d}^{-1}\left(
\varepsilon\right)  \notin\sigma(K_{0})$ and $\tilde{d}^{-1}\left(  0\right)
\in\sigma(K_{0})$ is an eigenvalue. Hence, \emph{H2 }replaces the
condition \emph{H0} in this setting. \par
We assume that the real-valued  functions
$\varepsilon\mapsto c(\varepsilon)$ and $\varepsilon\mapsto 
\rho(\varepsilon)$ have analytic extensions in an open neighborhood of the origin
$U\subseteq\mathbb{C}$. Obviously, this produce an analytic extension for $\varepsilon\mapsto \tilde{d}(  \varepsilon)$ as well. By abuse of notation, we denote such extensions by the same symbols.
\par
We introduce the analytic, operator valued map
\begin{equation}
U\ni\varepsilon\mapsto Q_{\kappa}(\varepsilon)  \equiv Q_{\kappa}(  c(
\varepsilon)  ,d(  \varepsilon)  )  =%
\begin{bmatrix}
c(  \varepsilon)  d(  \varepsilon)  +\left(  c\left(
\varepsilon\right)  d(  \varepsilon)  -1\right)  \kappa
^{2}N_{\kappa} & (c(  \varepsilon)  d(  \varepsilon)
-1)\kappa^{2}1_{\Omega}S\!L_{\kappa}\\
-\tilde{d}(  \varepsilon)  \,\gamma_{1}^{\mathrm{in}}N_{\kappa} &
1-\tilde{d}(  \varepsilon)  \,K_{\kappa}%
\end{bmatrix}
\label{Q_eps}%
\end{equation}
and look for a complex-valued function $\varepsilon\mapsto \kappa(\varepsilon)$, defined  in a complex neighborhood of the origin, such that 
\begin{equation}
\ker (Q_{\kappa(  \varepsilon)  }(  \varepsilon))
\neq\{  0\}  \,. \label{problem_eps}%
\end{equation}
Then, by Theorem \ref{teo-res} and Remark \ref{rem-res}, for $\varepsilon>0$ sufficiently small, $\varepsilon\mapsto\kappa^{2}(\varepsilon)$ gives rise to a branch of either resonances or negative eigenvalues of  $L(\varepsilon)\equiv L_{c(\varepsilon),d(\varepsilon)}$.\par
 From (\ref{eps_fun}) and
(\ref{d_tilda_eps}) there follows%
\[
\tilde{d}(  \varepsilon)  =\frac{1}{\lambda_{\circ}}-\frac{\rho_{0}%
}{\lambda_{\circ}^{2}}\,\varepsilon+\mathcal{O}(  \varepsilon)
\,,\qquad d(  \varepsilon)  =\frac{\lambda_{\circ}-1/2}{\lambda
_{\circ}+1/2}-\frac{\rho_{0}}{\left(  \lambda_{\circ}+1/2\right)  ^{2}}\,\varepsilon
+\mathcal{O}(  \varepsilon)  \,,
\]
\[
 g\left(  \varepsilon
\right)  :=c(  \varepsilon)  d(  \varepsilon)
=g_{0}+g_{1}\varepsilon+\mathcal{O}(  \varepsilon)  \,,
\]
where%
\[
g_{0}=c_{0}\frac{\lambda_{\circ}-1/2}{\lambda_{\circ}+1/2}\,,\qquad g_{1}=c_{1}%
\frac{\lambda_{\circ}-1/2}{\lambda_{\circ}+1/2}-\frac{c_{0}\rho_{0}}{\left(
\lambda_{\circ}+1/2\right)  ^{2}}\,,
\]
(notice that $b_{1}$ may be zero). By%
\[
Q_{0}(0)=%
\begin{bmatrix}
g_{0} & 0\\
-\frac{1}{\lambda_{\circ}}\,\gamma_{1}^{\mathrm{in}}N_{\kappa} & 1-\frac
{1}{\lambda_{\circ}}K_{\kappa}%
\end{bmatrix}
\,,
\]
the unique solution to $Q_{0}(0)(u,\varphi)=(0,0)$ is $(0,\varphi_{\circ})$. We
use the space $\hat{H}^{-1/2}(  \Gamma)  \simeq H^{-1/2}\left(
\Gamma\right)  $ introduced in Section \ref{Sec_Transmission} with the
modified scalar product (\ref{mod_prod}). The orthogonal projector on
$\varphi_{\circ}\in\hat{H}^{-1/2}(\Gamma)$ is the map%
\[
\pi_{\circ}\varphi:=\left\langle S_{0}\varphi_{\circ},\varphi\right\rangle _{\Gamma
}\varphi_{\circ}\,.
\]
Since $K_{0}$ is compact and self-adjoint on $\hat{H}^{-1/2}\left(
\Gamma\right)  $, one has%
\begin{equation}
\ran(  1-\lambda_{\circ}^{-1} K_{0})  =\left(  \ker(
1-\lambda_{\circ}^{-1} K_{0})  \right)  ^{\bot}=(  1-\pi_{\circ})
\hat{H}^{-1/2}(  \Gamma)  \,. \label{transmission_op_cond}%
\end{equation}
We adopt the notation%
\begin{equation}
\left(  1-\lambda_{\circ}^{-1} K_{0}\right)  _{\!\perp}:=(  1-\lambda
_{\circ}^{-1} K_{0}) |(  1-\pi_{\circ})  \hat
{H}^{-1/2}(  \Gamma)  =(  1-\pi_{\circ})  (
1-\lambda_{\circ}^{-1} K_{0})  (  1-\pi_{\circ})  \,.
\label{transmission_op_id}%
\end{equation}
By (\ref{transmission_op_cond}), $\left(  1-\lambda_{\circ}^{-1} K_{0}\right)
_{\!\perp}$ is a bounded bijection on $(  1-\pi_{\circ})  \hat
{H}^{-1/2}(  \Gamma)  $. Let%
\[
X:=L^{2}(  \Omega)  \oplus\hat{H}^{-1/2}(  \Gamma)
\]
and $P:X\rightarrow X$%
\[
P_{\circ}\left(  u,\varphi\right)  :=\left\langle \Psi_{\circ},\Psi\right\rangle
_{X}\Psi_{\circ}\,,\qquad\Psi_{\circ}:=\left(  \frac{1}{\lambda_{\circ}g_{0}}\ D\!L_{0}%
S_{0}\varphi_{\circ},\varphi_{\circ}\right)  \,,
\]
be the orthogonal projector on $\Psi_{\circ}$ in $X$. We define
\[
X_{\circ}:=\left\{  \Psi\in X:P_{\circ}\Psi=0\right\}  =\left(  1-P_{\circ}\right)  X\,.
\]

\begin{lemma}%
\begin{equation}
X_{\circ}:=\left\{  \left(  u,\varphi\right)  \in X:\pi_{\circ}\left(  \varphi
+\frac{1}{\lambda_{\circ}g_{0}}\gamma_{1}^{\mathrm{in}}N_{0}u\right)  =0\right\}
\label{X_def}%
\end{equation}

\end{lemma}

\begin{proof}
Let $\Psi\equiv\left(  u,\varphi\right)  $; then: $\Psi\in X_{\circ}$ if and only
if%
\begin{align*}
&  \left.  0=\left\langle \Psi_{\circ},\Psi\right\rangle _{X}=\frac{1}{\lambda
_{\circ}g_{0}}\left\langle D\!L_{0}S_{0}\varphi_{\circ},u\right\rangle _{\Omega
}+\left\langle S_{0}\varphi_{\circ},\varphi\right\rangle _{\Gamma}=\right. \\
&  \left.  \frac{1}{\lambda_{\circ}g_{0}}\left\langle S_{0}\varphi_{\circ},\gamma
_{1}^{\mathrm{in}}N_{0}u\right\rangle _{\Gamma}+\left\langle S_{0}\varphi
_{\circ},\varphi\right\rangle _{\Gamma}=\left\langle S_{0}\varphi_{\circ}%
,\varphi+\frac{1}{\lambda_{\circ}g_{0}}\gamma_{1}^{\mathrm{in}}N_{0}u\right\rangle
_{\Gamma}\,,\right.
\end{align*}
and hence, if and only if%
\[
\pi_{\circ}\left(  \varphi+\frac{1}{\lambda_{\circ}g_{0}}\gamma_{1}^{\mathrm{in}}%
N_{0}u\right)  =0\,.
\]

\end{proof}

Let us introduce the family of maps
\begin{align*}
&  \left.  \Phi_{\varepsilon,\kappa}:X\rightarrow\mathbb{C}\oplus
X_{\circ}\,,\right. \\
& \\
&  \left.  \Phi_{\varepsilon,\kappa}(u,\varphi):=\left(  \left\langle
S_{0}\varphi_{\circ},\varphi-\varphi_{\circ}\right\rangle _{\Gamma}\,,\ \left(
1-P_{\circ}\right)  Q_{\kappa}(\varepsilon)\left(  u,\varphi\right)  \right)
\,,\right.
\end{align*}
depending on $\varepsilon\in U\subseteq\mathbb{C}$ (an open neighbourhood of
the origin) and on $\kappa\in\mathbb{C}$. The assumptions \emph{H1} and the
mapping properties of the operators in $Q_{\kappa}(\varepsilon)$ imply that
$(  \varepsilon,\kappa)  \mapsto\Phi_{\varepsilon,\kappa}$ is
analytic in $U\times\mathbb{C}$.

\begin{lemma}
There exists an open neighbourhood $W_{0}$ of $(0,0)\in U\times\mathbb{C}$ and
a unique analytic map $\Psi(  \cdot,\cdot)  :W_{0}\rightarrow X$ such
that $\Psi(  0,0)  =(0,\varphi_{\circ})$ and%
\begin{equation}
\Phi_{\varepsilon,\kappa}(\Psi(  \varepsilon,\kappa)
)=0\,,\qquad(  \varepsilon,\kappa)  \in W_{0}\,. \label{AIFT_id}%
\end{equation}

\end{lemma}

\begin{proof}
Let us notice that $\Phi_{0,0}(0,\varphi_{\circ})=(0,0)$, while the Fr\'{e}chet
derivative of $\Phi_{0}$ at $(0,\varphi_{\circ})$ is%
\begin{align*}
&  \left.  D_{(0,\varphi_{\circ})}\Phi_{0,0}:X\rightarrow\mathbb{C}\oplus
X_{\circ}\,,\right. \\
& \\
&  \left.  \left(  D_{(0,\varphi_{\circ})}\Phi_{0,0}\right)  (u,\varphi)=\left(
\left\langle S_{0}\varphi_{\circ},\varphi\right\rangle _{\Gamma}\,,\ \left(
1-P_{\circ}\right)  \left(  g_{0}u,-\frac{1}{\lambda_{\circ}}\gamma_{1}^{\mathrm{in}%
}N_{0}u+(1-\frac{1}{\lambda_{\circ}}K_{0})\varphi\right)  \right)  \,.\right.
\end{align*}
Let us consider the problem%
\[
\left(  D_{(0,\varphi_{\circ})}\Phi_{0,0}\right)  (u,\varphi)=\left(  \kappa
_{\ast},u_{\ast},\varphi_{\ast}\right)  \in\mathbb{C}\oplus X_{\circ}\,.
\]
This corresponds to
\[
\left\{
\begin{array}
[c]{l}%
\left\langle S_{0}\varphi_{\circ},\varphi\right\rangle _{\Gamma}=\kappa_{\ast
}\,,\\
g_{0}u=u_{\ast}\,,\\
-\lambda_{\circ}^{-1}\gamma_{1}^{\mathrm{in}}N_{0}u+\left(  1-\frac{1}{\lambda
_{\circ}}K_{0}\right)  \varphi=\varphi_{\ast}\,,
\end{array}
\right.
\]
and, using $\hat{H}^{-1/2}(\Gamma)\ni\varphi=c_{\varphi}\varphi_{\circ}%
+\varphi_{\bot}$, one gets%
\begin{equation}
\left\{
\begin{array}
[c]{l}%
c_{\varphi}=\kappa_{\ast}\,,\\
g_{0}u=u_{\ast}\,,\\
(1-\lambda_{\circ}^{-1}K_{0})_{\bot}\varphi_{\bot}=\varphi_{\ast}+\frac{1}%
{\lambda_{\circ}g_{0}}\gamma_{1}^{\mathrm{in}}N_{0}u_{\ast}\,,
\end{array}
\right.  \label{range_pb}%
\end{equation}
where, by the previous Lemma,%
\[
\varphi_{\ast}+\frac{1}{\lambda_{\circ}g_{0}}\gamma_{1}^{\mathrm{in}}N_{0}u_{\ast
}\in(  1-\pi_{\circ})  \hat{H}^{-1/2}(\Gamma)\,.
\]
Since $(1-\lambda_{\circ}^{-1}K_{0})_{\bot}$ has bounded inverse on $\left(
1-\pi_{\circ}\right)  \hat{H}^{-1/2}(  \Gamma)  $, (\ref{range_pb}) is
uniquely solved by%
\begin{equation}
\left\{
\begin{array}
[c]{l}%
c_{\varphi}=\kappa_{\ast}\,,\\
u=\frac{1}{g_{0}}u_{\ast}\,,\\
\varphi_{\bot}=(1-\frac{1}{\lambda_{\circ}}K_{0})_{\bot}^{-1}\left(  \varphi
_{\ast}+\frac{1}{\lambda_{\circ}g_{0}}\gamma_{1}^{\mathrm{in}}N_{0}u_{\ast
}\right)  \,.
\end{array}
\right.  \label{range_pb_sol}%
\end{equation}
Hence, $D_{(0,\varphi_{\circ})}\Phi_{0,0}:X\rightarrow\mathbb{C}\oplus X_{\circ}$ is
an isomorphism whose inverse is%
\[
\left(  D_{(0,\varphi_{\circ})}\Phi_{0,0}\right)  ^{-1}(\kappa,u,\varphi)=\left(
g_{0}^{-1}u\,,\ \kappa\varphi_{\circ}+(1-\lambda_{\circ}^{-1}K_{0})_{\bot}%
^{-1}(\varphi+\frac{1}{\lambda_{\circ}g_{0}}\gamma_{1}^{\mathrm{in}}%
N_{0}u)\right)  \,.
\]
By the analytic implicit function theorem in Banach spaces, see e.g. in
\cite{Whitt PAMS 65}, there exists an open neighbourhood $W_{0}$ of $(0,0)\in
U\times\mathbb{C}$ and a unique analytic map $\Psi:W_{0}\rightarrow X$ such
that: $\Psi\left(  0,0\right)  =(0,\varphi_{\circ})$ and (\ref{AIFT_id}) holds.
\end{proof}

We next denote%
\[
\Psi(\varepsilon,\kappa)\equiv\left(  u(  \varepsilon,\kappa)
\,,\ \varphi(  \varepsilon,\kappa)  \right)  \,.
\]
Then,%
\begin{equation}
u\left(  0,0\right)  =0\,,\qquad\varphi\left(  0,0\right)  =\varphi_{\circ}\,,
\label{pivot}%
\end{equation}
while (\ref{AIFT_id}) corresponds to the relations%
\begin{align}
&  \left.  \left\langle S_{0}\varphi_{\circ},\varphi(\varepsilon,\kappa
)-\varphi_{\circ}\right\rangle _{\Gamma}=0\,,\right. \label{IFT 1}\\
& \nonumber\\
&  \left.  Q_{\kappa}(\varepsilon)(u(\varepsilon,\kappa),\varphi
(\varepsilon,\kappa))=f(\varepsilon,\kappa)\left(  \frac{1}{\lambda_{\circ}g_{0}%
}\,D\!L_{0}\varphi_{\circ},\varphi_{\circ}\right)  \,,\right.  \label{IFT 2}%
\end{align}
where%
\begin{align}
&  \left.  f(\varepsilon,\kappa):=\right. \nonumber\\
&  \left.  \frac{1}{\lambda_{\circ}g_{0}}\left\langle D\!L_{0}\varphi_{\circ},\left(
g(  \varepsilon)  +\left(  g(  \varepsilon)  -1\right)
\kappa^{2}N_{\kappa}\right)  u(  \varepsilon,\kappa)  +\left(
g(  \varepsilon)  -1\right)  \kappa^{2}SL_{\kappa}\varphi\left(
\varepsilon,\kappa\right)  \right\rangle _{\Omega}\right. \nonumber\\
&  \left.  +\left\langle S_{0}\varphi_{\circ},-\tilde{d}\left(  \varepsilon
\right)  \gamma_{1}^{\mathrm{in}}N_{\kappa}u(  \varepsilon,\kappa)
+\left(  1-\tilde{d}(  \varepsilon)  K_{\kappa}\right)
\varphi(  \varepsilon,\kappa)  \right\rangle _{\Gamma}\,,\right.
\label{f_def}%
\end{align}
is analytic from $W_{0}$ to $\mathbb{C}$. By (\ref{IFT 1}) one has%
\[
\left\langle S_{0}\varphi_{\circ},\left(  \partial_{\varepsilon}^{m}%
\partial_{\kappa}^{n}\varphi\right)  \left(  0,0\right)  \right\rangle
_{\Gamma}=0\,,\qquad n+m\geq1\,,
\]
from which it follows%
\begin{equation}
\partial_{\varepsilon}^{m}\partial_{\kappa}^{n}\varphi\left(  0,0\right)
\in(  1-\pi_{\circ})  \hat{H}^{-1/2}(  \Gamma)  \,,
\label{var_phi_constraint}%
\end{equation}
while (\ref{IFT 2}) yields%
\begin{equation}
\left.  \partial_{\varepsilon}^{m}\partial_{\kappa}^{n}Q_{\kappa}%
(\varepsilon)(u(\varepsilon,\kappa),\varphi(\varepsilon,\kappa))\right\vert
_{\varepsilon=\kappa=0}=\partial_{\varepsilon}^{m}\partial_{\kappa}%
^{n}f\left(  0,0\right)  \,\left(  \frac{1}{\lambda_{\circ}g_{0}}\,D\!L_{0}%
\varphi_{\circ},\varphi_{\circ}\right)  \,. \label{var_Psi_constraint}%
\end{equation}
We will use the following result.

\begin{lemma}%
\begin{equation}
S\!L_{0}\varphi_{\circ}=\frac{2}{2\lambda_{\circ}-1}DL_{0}S_{0}\varphi_{\circ}
\label{SL_0_id}%
\end{equation}

\end{lemma}

\begin{proof}
Let $u=S\!L_{0}\varphi_{\circ}$; then%
\[
\left\{
\begin{array}
[c]{l}%
-\Delta u=0\,,\qquad\text{in }\Omega\,,\\
\gamma_{0}^{\mathrm{in}}u=S_{0}\varphi_{\circ}\,.
\end{array}
\right.
\]
From \cite[Theorem 7.5]{McLe} follows%
\[
u=-DL_{0}S_{0}\varphi_{\circ}+S\!L_{0}S_{0}^{-1}\left(  \frac{1}{2}+N\!P_{0}\right)
S_{0}\varphi_{\circ}\,.
\]
From $N\!P_{0}S_{0}\varphi_{\circ}=\lambda_{\circ}S_{0}\varphi_{\circ}$ follows%
\[
S\!L_{0}\varphi_{\circ}=-DL_{0}S_{0}\varphi_{\circ}+\left(  \frac{1}{2}+\lambda
_{\circ}\right)  S\!L_{0}\varphi_{\circ}%
\]
where, by \emph{H2}, $\lambda_{\circ}\in\left(  -1/2,1/2\right)  \backslash
\left\{  0\right\}  $. Therefore, we get%
\[
S\!L_{0}\varphi_{\circ}=-\left(  \frac{1}{2}-\lambda_{\circ}\right)  ^{-1}DL_{0}%
S_{0}\varphi_{\circ}%
\]
which corresponds to (\ref{SL_0_id}).
\end{proof}

Introducing the coefficients%
\[
c_{n,m}:=\left(  \partial_{\varepsilon}^{n}\partial_{\kappa}^{m}f\right)
\left(  0,0\right)  \,,
\]
one has

\begin{lemma}%
\begin{align}
&  \left.  c_{1,0}=\frac{\rho_{0}}{\lambda_{\circ}}\left\Vert \varphi
_{\circ}\right\Vert _{\hat{H}^{-1/2}(  \Gamma)  }^{2}\,,\right.
\label{c_10}\\
&  \left.  c_{0,1}=c_{1,1}=c_{2,1}=c_{3,0}=0\,,\right. \label{c01_11}\\
&  \left.  c_{0,2}=4\,\frac{g_{0}-1}{g_{0}}\frac{1}{2\lambda_{\circ}-1}\left\Vert
DL_{0}S_{0}\varphi_{\circ}\right\Vert _{L^{2}(  \Omega)  }%
^{2}\,,\right.  \label{c_02}%
\end{align}

\end{lemma}

\begin{proof}
Using (\ref{f_def}) one gets%
\begin{align*}
&  \left.  c_{1,0}=\frac{1}{\lambda_{\circ}}\left\langle D\!L_{0}\varphi
_{\circ},\partial_{\varepsilon}u\left(  0,0\right)  \right\rangle _{\Omega}%
-\frac{1}{\lambda_{\circ}}\left\langle S_{0}\varphi_{\circ},\gamma_{1}^{\mathrm{in}%
}N_{0}\partial_{\varepsilon}u\left(  0,0\right)  \right\rangle _{\Gamma
}\right. \\
&  \left.  +\left\langle S_{0}\varphi_{\circ},\right.  \left(  1-\frac{1}%
{\lambda_{\circ}}K_{0}\right)  \partial_{\varepsilon}\left.  \varphi\left(
0,0\right)  \right\rangle _{\Gamma}+\frac{\rho_{0}}{\lambda_{\circ}}\left\langle
S_{0}\varphi_{\circ},\varphi_{\circ}\right\rangle _{\Gamma}\right. \\
&  \left.  =\frac{1}{\lambda_{\circ}}\left\langle S_{0}\varphi_{\circ},\gamma
_{1}^{\mathrm{in}}N_{0}\partial_{\varepsilon}u\left(  0,0\right)
\right\rangle _{\Gamma}-\frac{1}{\lambda_{\circ}}\left\langle S_{0}\varphi
_{\circ},\gamma_{1}^{\mathrm{in}}N_{0}\partial_{\varepsilon}u\left(  0,0\right)
\right\rangle _{\Gamma}\right. \\
&  \left.  +\left\langle {}\right.  \left(  1-\frac{1}{\lambda_{\circ}}%
N\!P_{0}\right)  S_{0}\varphi_{\circ},\partial_{\varepsilon}\left.  \varphi\left(
0,0\right)  \right\rangle _{\Gamma}+\frac{\rho_{0}}{\lambda_{\circ}}\left\langle
S_{0}\varphi_{\circ},\varphi_{\circ}\right\rangle _{\Gamma}\right. \\
&  \left.  =\frac{\rho_{0}}{\lambda_{\circ}}\left\langle S_{0}\varphi_{\circ}%
,\varphi_{\circ}\right\rangle _{\Gamma}=\frac{\rho_{0}}{\lambda_{\circ}}\left\Vert
\varphi_{\circ}\right\Vert _{\hat{H}^{-1/2}(  \Gamma)  }^{2}\,.\right.
\end{align*}

A similar computation yields%
\[
\left.  c_{0,1}=\frac{1}{\lambda_{\circ}}\left\langle D\!L_{0}\varphi_{\circ}%
,N_{0}\partial_{\kappa}u\left(  0,0\right)  \right\rangle _{\Omega}-\frac
{1}{\lambda_{\circ}}\left\langle S_{0}\varphi_{\circ},\gamma_{1}^{\mathrm{in}}%
N_{0}\partial_{\kappa}u\left(  0,0\right)  \right\rangle _{\Gamma
}=0\,.\right.
\]
Using (\ref{var_Psi_constraint}) and (\ref{var_phi_constraint}) with $m=0$ and
$n=1$, this yields%
\[
\left.  \partial_{\kappa}Q_{\kappa}(\varepsilon)(u(\varepsilon,\kappa
),\varphi(\varepsilon,\kappa))\right\vert _{\varepsilon=\kappa=0}=%
\begin{pmatrix}
g_{0}\partial_{\kappa}u\left(  0,0\right) \\
\left(  1-\frac{1}{\lambda_{\circ}}K_{0}\right)  _{\bot}\partial_{\kappa}%
\varphi\left(  0,0\right)
\end{pmatrix}
=%
\begin{pmatrix}
0\\
0
\end{pmatrix}
\,,
\]
and hence%
\begin{equation}%
\begin{pmatrix}
\partial_{\kappa}u\left(  0,0\right) \\
\partial_{\kappa}\varphi\left(  0,0\right)
\end{pmatrix}
=%
\begin{pmatrix}
g_{0}^{-1} & \\
& \left(  1-\frac{1}{\lambda_{\circ}}K_{0}\right)  _{\bot}^{-1}%
\end{pmatrix}%
\begin{pmatrix}
0\\
0
\end{pmatrix}
=%
\begin{pmatrix}
0\\
0
\end{pmatrix}
\,. \label{D_01_id}%
\end{equation}

From%
\[
N_{0}^{\left(  1\right)  }u=\left.  \partial_{\kappa}N_{\kappa}\right\vert
_{\kappa=0}u=1_{\Omega}\left(  \left.  \partial_{\kappa}R_{\kappa}\right\vert
_{\kappa=0}\right)  1_{\Omega}^{\ast}u=\frac{i}{4\pi}\left\langle
1,u\right\rangle _{\Omega}%
\]
follows%
\begin{equation}
\gamma_{1}^{\mathrm{in}}N_{0}^{\left(  1\right)  }=0\,. \label{var_N_id}%
\end{equation}
Using (\ref{pivot}), (\ref{D_01_id}) and (\ref{var_N_id}) one gets%
\[
\left.  c_{1,1}=-\frac{1}{\lambda_{\circ}}\left\langle S_{0}\varphi_{\circ},\gamma
_{1}^{\mathrm{in}}N_{0}^{\left(  1\right)  }\partial_{\varepsilon}u\left(
0,0\right)  \right\rangle _{\Gamma}=0\,.\right.
\]
Using (\ref{var_Psi_constraint}) and (\ref{var_phi_constraint}) with $m=n=1$,
this yields%
\[
\left.  \partial_{\varepsilon}\partial_{\kappa}Q_{\kappa}(\varepsilon
)(u(\varepsilon,\kappa),\varphi(\varepsilon,\kappa))\right\vert _{\varepsilon
=\kappa=0}=%
\begin{pmatrix}
g_{0}\partial_{\varepsilon}\partial_{\kappa}u\left(  0,0\right) \\
\left(  1-\frac{1}{\lambda_{\circ}}K_{0}\right)  _{\bot}\partial_{\varepsilon
}\partial_{\kappa}\varphi\left(  0,0\right)
\end{pmatrix}
=%
\begin{pmatrix}
0\\
0
\end{pmatrix}
\,,
\]
which implies%
\begin{equation}%
\begin{pmatrix}
\partial_{\varepsilon}\partial_{\kappa}u\left(  0,0\right) \\
\partial_{\varepsilon}\partial_{\kappa}\varphi\left(  0,0\right)
\end{pmatrix}
=%
\begin{pmatrix}
0\\
0
\end{pmatrix}
\,. \label{D_11_id}%
\end{equation}

For $c_{0,2}$ we get%
\[
c_{0,2}=2\,\frac{g_{0}-1}{\lambda_{\circ}g_{0}}\left\langle DL_{0}S_{0}\varphi
_{\circ},S\!L_{0}\varphi_{\circ}\right\rangle _{\Omega}%
\]
and using (\ref{SL_0_id}), we obtain (\ref{c_02}).

Finally $c_{2,1}=0$ follows using (\ref{pivot}), (\ref{D_01_id}) and
(\ref{D_11_id}) and proceeding as above, while for $c_{0,3}\ $one has%
\[
c_{0,3}=\frac{i}{\pi}\frac{g_{0}-1}{\lambda_{\circ}g_{0}}\left\langle
1,\varphi_{\circ}\right\rangle _{\Gamma}\left\langle DL_{0}S_{0}\varphi
_{\circ},1\right\rangle _{\Omega}%
\]
and $c_{0,3}=0$ follows from (\ref{zero_mean}).
\end{proof}

Using (\ref{c01_11}), the analytic expansion of $f$ at $\left(  0,0\right)  $
reads as%
\begin{equation}
f(  \varepsilon,\kappa)  =c_{1,0}\varepsilon+\frac{1}{2}\,%
c_{2,0}\varepsilon^{2}+\frac{1}{2}\,c_{0,2}\kappa^{2}+\frac{1}{6}\,c_{1,2}%
\varepsilon\kappa^{2}+F(  \varepsilon,\kappa)  \,, \label{f_exp}%
\end{equation}
where $c_{1,2}$ can be explicitly determined in terms of the parameters
(following similar computations as the ones above). Since this not affects the
next results, we omit its computation for the sake of simplicity. The function
$F$ is analytic and
\begin{equation}
F(  \varepsilon,\kappa)  =O(  \varepsilon
^{3})  +O(  (\varepsilon^{2}+\kappa^{2})^{2})  \,.
\label{f_exp_rem}%
\end{equation}

\begin{theorem}\label{Teo-res}
There exists $\varepsilon_{0}>0$ and a unique couple of
functions $(0,\varepsilon_{0})\ni\varepsilon\mapsto\kappa_{\pm}(  \varepsilon)  $, analytic with respect to  $\sqrt{\varepsilon}$, 
\begin{equation}\label{k_plusmin_exp}
\kappa_{\pm}(\varepsilon)=\pm z_{0}\sqrt{\varepsilon}\,(1+O(  \varepsilon))\pm \sqrt{\varepsilon}\,\xi(\pm\sqrt{\varepsilon}\,)\,,\qquad |\xi(t)|\lesssim |t|^{3}
\end{equation}
such that
$$
\ker(Q_{\kappa_{\pm}(\varepsilon)}(\varepsilon))\not=\{0\}\,.
$$
Here, $\pm z_{0}$ are the two roots of the equation
$$
\frac{z^{2}}2+\frac{\rho_{0}}{\lambda_{\circ}}\left\Vert \varphi
_{\circ}\right\Vert _{\hat{H}^{-1/2}(  \Gamma)  }^{2}=0\,.
$$
\end{theorem}

\begin{proof}
Let $\left(  t,z\right)  $ be such that $\left(  t^{2},tz\right)  \in W_{0}$
and define
\begin{equation}\label{f-g}
g(t,z):=t^{-2}f(t^{2},tz)\,.
\end{equation}
By (\ref{f_exp}), there follows%
\[
g\left(  t,z\right)  =c_{1,0}+\frac{1}{2}c_{2,0}t^{2}+\frac{1}{2}%
c_{0,2}\,z^{2}+\frac{1}{6}c_{1,2}\,t^{2}z^{2}+\mathcal{O}\left(  t^{4}\right)
+\mathcal{O}\left(  t^{2}(t^{2}+z^{2})^{2}\right)
\]
and $g$ is analytic. Let 
$$
z_{\pm}=\pm z_{0}\,,\qquad z_{0}=\begin{cases}\sqrt{2\,|c_{1,0}|}&c_{1,0}<0\\
i\sqrt{2\,c_{1,0}}&c_{1,0}>0\,,
\end{cases}
$$
 be the two roots of the equation
\[
c_{1,0}+\frac{1}{2}\,z^{2}=0\,.
\]
One has
$$
g(0,z_{\pm})=0\,,\qquad\partial_{z}g(0,z_{\pm})=c_{0,2}\,z_{\pm}%
\not =0\,,\qquad
$$
$$
\partial_{t}g(0,z_{\pm})=0\,,\qquad\partial_{t}^{2}g(0,z_{\pm
})=c_{2,0}+\frac{1}{3}\,c_{1,2}\,z_{\pm}^{2}=c_{2,0}-\frac{2}{3}%
\,c_{1,0}c_{1,2}\,.
$$
By the implicit function theorem, there exists two unique analytic functions
$$
t\mapsto z_{\pm}(t)=z_{\pm}-\frac{\partial_{t}^{2}g(0,z_{\pm})}%
{2\partial_{z}g(0,z_{\pm})}\,t^{2}+\xi_{\pm }(t)\,,\qquad |\xi_{\pm }(t)|\lesssim |t|^{3}\,,
$$
such that
\[
g(t,z_{\pm}(t))=0\,.
\]
By $g(-t,z)=g(t,-z)$, one has $g(t,-z_{\pm}(-t))=0$ and so, by $z_{\pm}(0)=\pm z_{0}$ and by unicity, one gets
$$
z_{-}(t)=-z_{+}(-t)\,.
$$
Therefore, 
$$
z_{\pm}(t)=\pm z(\pm t)\,,
$$
where
\begin{align*}
  z(t)=&z_{0}\left(  1-\frac{1}{2c_{2,0}z_{0}^{2}}\left(
c_{2,0}-\frac{2}{3}\,c_{1,0}c_{1,2}\right)  t^{2}\right)  +\xi(t)\\
  =&z_{0}\left(  1+\frac{1}{4c_{2,0}c_{1,0}}\left(  c_{2,0}-\frac
{2}{3}\,c_{1,0}c_{1,2}\right)  t^{2}\right)  +\xi(t)\\
 =&z_{0}\left(  1-\frac{1}{4}\left(  \frac{1}{c_{1,0}}-\frac{2}%
{3}\,\frac{c_{1,2}}{c_{2,0}}\right)  t^{2}\right)  +\xi(t)\,,\qquad |\xi(t)|\lesssim |t|^{3}\,.
\end{align*}
By \eqref{f-g},  for $\varepsilon>0$ sufficiently small,
\[
k_{\pm}(\varepsilon)=\pm\sqrt{\varepsilon}z(\pm\sqrt{\varepsilon}\,)=\pm z_{0}%
\sqrt{\varepsilon}\left(  1-\frac{1}{4}\left(  \frac{1}{c_{1,0}}-\frac{2}%
{3}\,\frac{c_{1,2}}{c_{2,0}}\right)  \varepsilon\right)  \pm\sqrt\varepsilon \,\xi(\pm\sqrt\varepsilon\,)
\]
solves
\[
f(\varepsilon,k_{\pm}(\varepsilon))=0\,.
\]
\end{proof}

\begin{remark}\label{Rem-res}
Case 1) If $\rho_{0}/\lambda_{\circ}<0$ then 
$$
\kappa_{\pm}(\varepsilon)=
\pm \sqrt{2}\,\left|{\rho_{0}}/
\lambda_{\circ}\right|^{1/2} \|\varphi_{\circ}\| _{\hat{H}^{-1/2}(
\Gamma)  }\sqrt{\varepsilon}\,(1+O(  \varepsilon))\pm \sqrt{\varepsilon}\,\xi(\pm\sqrt{\varepsilon}\,)
$$
Case 2) If $\rho_{0}/\lambda_{\circ}>0$ then 
$$
\kappa_{\pm}(\varepsilon)=
\pm i\sqrt{2}\,\left({\rho_{0}}/
\lambda_{\circ}\right)^{1/2} \|\varphi_{\circ}\| _{\hat{H}^{-1/2}(
\Gamma)  }\sqrt{\varepsilon}\,(1+O(  \varepsilon))\pm \sqrt{\varepsilon}\,\xi(\pm\sqrt{\varepsilon}\,)
$$
We take in account Remark \ref{rem-res}. In Case 1, since there are no positive eigenvalues, both $\kappa^{2}_{\pm}(\varepsilon)$ are resonances.
In Case 2, $\kappa^{2}_{-}(\varepsilon)$ is a resonance and $\kappa^{2}_{+}(\varepsilon)$ is a negative eigenvalue. Hence, one has the following scenery:

\item{case $\lambda_{\circ}>0$)}  the shift of  $\widetilde d^{-1}$ from $\lambda_{\circ}$ to the right ($\rho_{0}<0$)  generates two branches of resonances,  emanating from the zero eigenvalue. Conversely, the shift of  $\widetilde d^{-1}$ from $\lambda_{\circ}$ to the left ($\rho_{0}>0$)  generates two branches, one consisting of negative eigenvalues and the other of resonances. 

\item{case $\lambda_{\circ}<0$)} the previous results are reversed: the shift to the left   generates two branches of resonances and the shift to the right generates one branch consisting of negative eigenvalues and another one consisting of resonances.
\vskip5pt\noindent
Summing up: if $\widetilde d^{-1}(\varepsilon)/\lambda_{\circ}>1$, then two branches of resonances are created. Otherwise,  if $\widetilde d^{-1}(\varepsilon)/\lambda_{\circ}<1$,  then two branches, one consisting of negative eigenvalues and the other of resonances, are created.
 \end{remark}

\end{document}